\documentclass{article}
\usepackage{fontspec}
\usepackage{amsmath,amssymb,amsthm,bbm,float,graphicx,geometry,mathtools,parskip,setspace,subcaption}
\usepackage{mathrsfs}
\usepackage{enumerate}
\usepackage{nicefrac}
\usepackage{todonotes}
\usepackage{comment}
\usepackage[colorlinks, linkcolor = red!80!black, citecolor = blue!80!black, breaklinks, pdfauthor={Lukas Luechtrath, Christian Moench}]{hyperref}
\usepackage{orcidlink}
\usepackage{titling}
\usepackage{tikz}
\usetikzlibrary{backgrounds}
\usetikzlibrary{patterns}
\usetikzlibrary{positioning, shapes.geometric}

\usepackage{caption}
\usepackage{graphicx}
\allowdisplaybreaks%breaks long equations onto multiple pages is needed
\newcommand{\N}{\mathbb{N}}

\newtheorem{theorem}{Theorem}[section]
\newtheorem{lemma}[theorem]{Lemma}

\newtheorem{prop}[theorem]{Proposition}

\theoremstyle{definition}

\newtheorem{remark}[theorem]{Remark}

\renewcommand{\P}{\mathbb{P}}
\newcommand{\x}{\boldsymbol{x}}
\newcommand{\y}{\boldsymbol{y}}
\newcommand{\1}{\mathbbm{1}}
\newcommand{\G}{\mathscr{G}}
\newcommand{\C}{\mathscr{C}}
\newcommand{\X}{\boldsymbol{X}}
\newcommand{\E}{\mathbb{E}}
\newcommand{\cE}{\mathcal{E}}
\newcommand{\0}{\boldsymbol{o}}
\newcommand{\R}{\mathbb{R}}
\renewcommand{\d}{\mathrm{d}}
\newcommand{\z}{\boldsymbol{z}}
\newcommand{\w}{\boldsymbol{w}}
\renewcommand{\v}{\boldsymbol{v}}
\renewcommand{\u}{\boldsymbol{u}}
\newcommand{\Z}{\mathbb{Z}}
\newcommand{\Y}{\boldsymbol{Y}}

\newcommand{\cX}{\mathcal{X}}
\newcommand{\cG}{\mathcal{G}}
\newcommand{\D}{\mathscr{D}}
\newcommand{\cD}{\mathcal{D}}
\newcommand{\bbT}{\mathbb{T}}
\newcommand{\scrV}{\mathscr{V}}
\newcommand{\scrE}{\mathscr{E}}
\newcommand{\scrX}{\mathscr{X}}
\renewcommand{\o}{\0}
\newcommand{\scrN}{\mathscr{N}}

\renewcommand{\d}{\,\mathrm{d}}
\newcommand{\scrI}{\mathscr{I}}
\newcommand{\cN}{\bfN}

\newcommand{\Var}{\textup{Var}}
\newcommand{\bfN}{\mathbf{N}}
\newcommand{\bfM}{\mathbf{M}}
\newcommand{\cZ}{\mathcal{Z}}

\pretitle{\centering\LARGE\scshape}
 \posttitle{\vskip 0.75cm}

 \predate{\vskip 0.5 cm \centering\large}
 \postdate{\par}

\usepackage[style = numeric, sorting=nyt, url = false, abbreviate=false, maxbibnames=9, sortcites=true, doi = true, backend = biber, giveninits = true, isbn = false]{biblatex}

\renewbibmacro{in:}{\ifentrytype{article}{}{\printtext{\bibstring{in}\intitlepunct}}}

\DeclareFieldFormat{journaltitle}{\mkbibemph{#1\isdot}} 

\renewbibmacro*{volume+number+eid}{%
  \printfield{volume}%
  \iffieldundef{number}
    {}
    {\setunit*{:}\printfield{number}}%
  \iffieldundef{eid}
    {}
    {\setunit{\addcomma\space}\printfield{eid}}%
}

\bibliography{bib.bib}

\title{Age-dependent random connection models with arc reciprocity: clustering and connectivity\thanks{A preliminary version of this work appeared in the proceedings of the 19th Workshop on \emph{Modelling and Mining Networks} (WAW 2024)~\cite{WAW24}}}

\thanksmarkseries{arabic}

\author{
Lukas L\"{u}chtrath \orcidlink{0000-0003-4969-806X}\thanks{Weierstrass Institute for Applied Analysis and Stochastics, Anton-Wilhelm-Amo-Str.\ 39, 10117 Berlin, Germany} \\ lukas.luechtrath@wias-berlin.de \\
\and
Christian M\"{o}nch \orcidlink{0000-0002-6531-6482}%\thanks{Johannes Gutenberg-Universität Mainz, Staudingerweg 9, 55128 Mainz, Germany}
\\ cmoench25@gmail.com
}

\date{March 17, 2026}

\begin{document}
\maketitle

\begin{spacing}{0.9}
\begin{abstract} 
\noindent We introduce a model for directed spatial networks. Starting from an age-based preferential attachment model in which all arcs point from younger to older vertices, we add \emph{reciprocal} connections whose probabilities depend on the age difference between their end-vertices. This yields a directed graph with reciprocal correlations, a power-law indegree distribution, and a tunable outdegree distribution. We consider two versions of the model: an infinite version embedded in \(\R^d\), which can be constructed as a weight-dependent random connection model with a non-symmetric kernel, and a growing sequence of graphs on the unit torus that converges locally to the infinite model. Besides establishing the local limit result linking the two models, we investigate degree distributions, various directed clustering metrics, and directed percolation.

\smallskip
\noindent\footnotesize{{\textbf{AMS-MSC 2020}: 05C80 (primary), 60K35, 05C82 (secondary)}

}
\end{abstract}
\end{spacing}

%%%%%%%%%%%%%%%%%%%%
%%%% Introduction %%
%%%%%%%%%%%%%%%%%%%%

\section{Introduction and overview of results}\label{sec:Intro}
Directed networks in data science often arise from asymmetric interactions in which the probability of forming a link in one direction is shaped by how attractive the corresponding reverse interaction would have been. A canonical example is provided by social media platforms such as \emph{X} (formerly \emph{Twitter}), \emph{Instagram}, or \emph{YouTube}, whose follower and subscription graphs exhibit pronounced directionality, strong heterogeneity, and substantial but highly non-uniform reciprocation \cite{KwakLPM10,MyersSGL14,WattenhoferWZ12}. Users follow accounts of interest, but reciprocal connections occur in a markedly influence-dependent way: high-profile accounts attract many followers yet follow back only rarely, and the chance of reciprocation depends not only on global popularity but also on topical similarity, geographical proximity, and local patterns of engagement. Such dynamics generate local clustering, correlated reciprocation, and heavy-tailed indegree distributions, features widely observed in empirical directed data and closely tied to influence and attention effects \cite{ChaHBG10,KwakLPM10}. More broadly, reciprocity in real directed networks is consistently and significantly non-random across domains ranging from biological to social and economic networks \cite{GarlascheLliLoffredo2004}.

Capturing these effects in a mathematically transparent and analytically tractable manner remains a challenge. Classical spatial models and age-based preferential attachment mechanisms reproduce heterogeneity and geometry, but they typically do not model directed arc formation beyond the temporal ordering. As a consequence, they miss structural phenomena driven by reciprocal interaction, such as asymmetric but correlated neighbourhoods, elevated motif frequencies, and non-vanishing directed clustering. These are precisely the features that matter in applications to influence estimation, link prediction, recommendation systems, and the analysis of community and motif structure in directed graphs, where reciprocity is a central modelling primitive rather than a negligible correction \cite{YuZYK13,DurakKoldaPinarSeshadhri2012}.

\paragraph{Our contribution}
We analyse a spatial directed generative model that incorporates an explicit \emph{reciprocity mechanism} which was proposed by us in \cite{WAW24}. Each vertex is marked with a spatial position and a birth time, and attempts to connect to earlier vertices according to an age-dependent kernel. The distinctive feature of our model is that the probability of forming an arc $\x\to \y$ between two vertices $\x$ and $\y$ depends not only on the conventional spatial and temporal structure, but also on how attractive $\x$ would have been to $\y$ had the interaction occurred in the opposite direction. This reciprocal coupling induces correlation between incoming and outgoing processes while retaining full analytical tractability through a Poisson spatial embedding.

From a data-science perspective, this model serves as an interpretable generative mechanism for directed networks with realistic reciprocity effects. The parameters governing spatial decay, preference, and reciprocal attractiveness directly control local and global structure, making the model a promising candidate for simulation studies, synthetic data generation, and principled statistical inference for networks in which asymmetric but mutually influenced interactions are prominent.

\begin{figure}[t]
    \centering
    \includegraphics[width=0.75\textwidth]{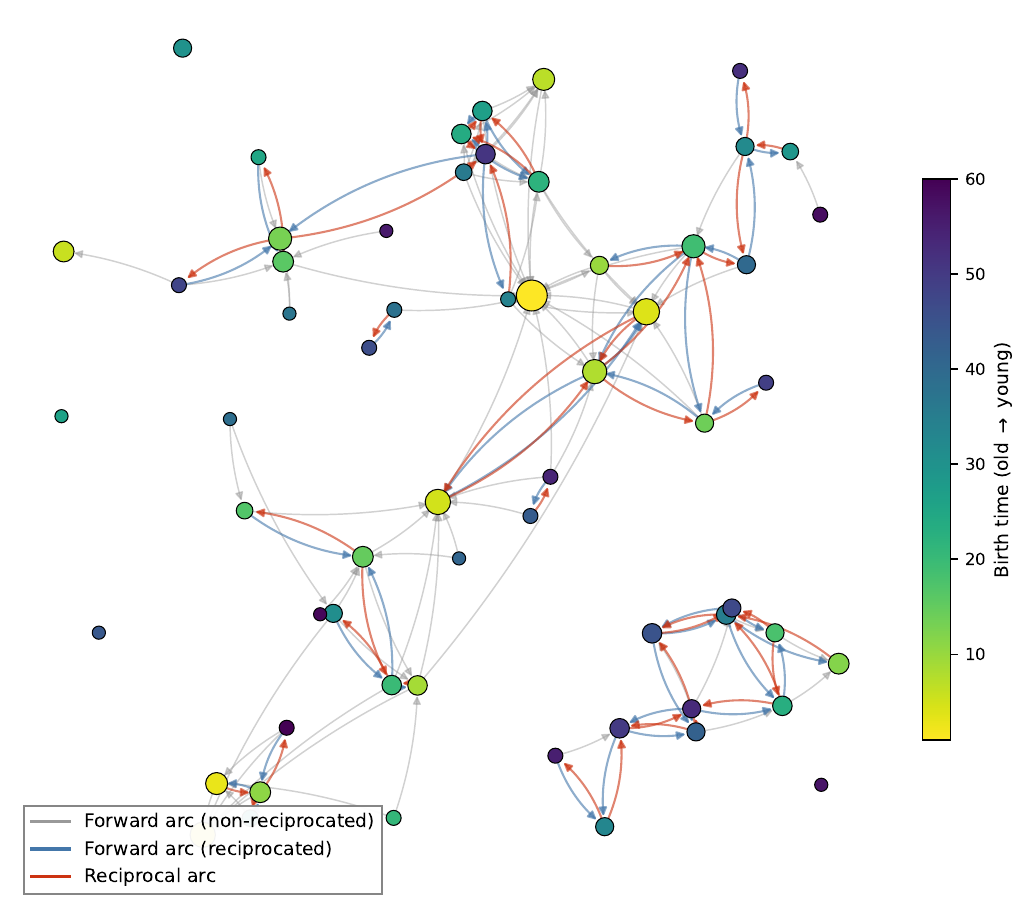}
    \caption{Simulation of a finitary toroidal variant of the DARCM (cf.\ Section~\ref{sec:DPA}) on the unit torus $[-\nicefrac{1}{2},\nicefrac{1}{2})^2$ with $N=60$ vertices, $\beta=0.4$, $\gamma=0.35$, $\delta=2.5$, and $\Gamma=1$. Vertex size is proportional to indegree; colour indicates birth time (dark = old, light = young). Gray arcs are non-reciprocated forward arcs (younger $\to$ older); blue arcs are forward arcs that were reciprocated; red arcs are the corresponding reciprocal arcs. Torus-wrapping arcs are suppressed for clarity.}
    \label{fig:simulation}
\end{figure}

\paragraph{Summary of results.}
Our analysis addresses three structural aspects essential for modelling real-world directed networks:

\begin{enumerate}
    \item \emph{Local connectivity and degree structure.}  
    We derive explicit formulae and asymptotics for the indegree and outdegree distributions. The indegree exhibits a power-law tail arising from temporal effects, while the outdegree may be either light-tailed or heavy-tailed, depending on the strength of the reciprocity. These results allow for the prediction of influential nodes and hub formation.

    \item \emph{Reciprocity-induced clustering.}  
    Motivated by common data-science metrics, we investigate two directed clustering quantities: a friend clustering coefficient and an interest clustering coefficient. Using Palm calculus for the underlying Poisson process, we show under which conditions they remain strictly positive in the sparse limit. This behaviour matches empirical observations in follower networks but sharply contrasts with classical preferential attachment and with the standard \emph{age-dependent random connection model}~\cite{GGLM2019}, for which directed clustering typically vanishes.

    \item \emph{Global geometry and reachability.}
    We further study directed percolation-type phenomena and accessibility properties, demonstrating how reciprocal interaction shapes the emergence of large weakly connected components and influences global navigation in the network. In particular, for sufficiently heavy-tailed degree distributions the critical edge intensity $\overset{\rightarrow}{\beta}_c$ vanishes, meaning that information cascades and viral spreading can occur globally even under high local failure rates whenever degree heterogeneity is strong enough.
\end{enumerate}

\paragraph{Relevance to network and data science.}
Reciprocal effects are among the most important structural drivers in directed network data, yet few analytically tractable models incorporate them explicitly. Our model offers:
\begin{itemize}
    \item a principled generative mechanism for directed graphs with tunable reciprocity,
    \item interpretable parameters that connect naturally to influence, locality, and user similarity,
    \item mathematical foundations that can inform inference procedures, simulation benchmarks, or structural analyses for machine learning on graphs.
\end{itemize}

\paragraph{Organisation of the paper.}
We define the model and its parameters in Section~\ref{sec:Model}. Sections~\ref{sec:degrees} and~\ref{sec:perc} analyse local degree structure and out-percolation, respectively.
In Section~\ref{sec:DPA}, we relate the infinite Poissonian model to a growing directed preferential attachment construction via weak local limits, and we deduce sparsity and empirical degree asymptotics.
Section~\ref{sec:cluster} studies reciprocity-driven clustering through friend and interest clustering coefficients. We conclude in Section~\ref{sec:outlook} with a discussion of open problems and future directions, including graph distances, strong percolation phenomena, statistical fitting to real network data, and the analysis of ranking measures such as PageRank under tunable reciprocity. Proofs are collected in Section~\ref{sec:proofs}.

%%%%%%%%%%%%%%%%%%%%
%%%%% Model %%%%%%%%
%%%%%%%%%%%%%%%%%%%%
\section{Model description} \label{sec:Model}
We begin with a description of the directed age-dependent random connection model (DARCM) as an infinite geometric digraph, which is the central subject studied in this work. We elaborate in Sections~\ref{sec:LLN} how this digraph appears as the weak local limit~\cite{BenjaminiSchramm2001} of a preferential-attachment-type sequence of growing networks that we introduce in Section~\ref{sec:DPA}; see Figure~\ref{fig:simulation} for a simulation of the finite version. Throughout, a \emph{digraph} $D$ is a countable collection of vertices $\scrV(D)$ together with a set 
\[
    \scrE(D)\subset\big\{(u,v)\in \scrV(D)^2\colon u\neq v\big\},
\]
indicating the arcs in $D$. A \emph{geometric digraph} is formally defined as a digraph $D$ together with a location map $\operatorname{loc}:\scrV(D)\to M$, where $M$ is some metric space. The location map encodes information of the spatial position of the vertices in the space $M$. In this paper, $M$ is always either $\R^d$ or a $d$-dimensional torus of finite diameter, equipped with the Euclidean metric or the torus metric, respectively.  

The vertex set of the \emph{directed age-dependent random connection model} (DARCM) is a unit intensity Poisson process \(\scrX\) on \(\R^d\times (0,1)\). We usually view \(\scrX\) as a marked Poisson point process~\cite{LastPenrose2017} on \(\R^d\) in which each point is assigned an independent mark uniformly distributed on \((0,1)\). We denote the vertices by \(\x=(x,t_x)\in\scrX\) and call \(x\in\R^d\) the vertex' \emph{location} and \(t_x\in(0,1)\) the vertex' \emph{birth time}. Hence, the map $\operatorname{loc}(\cdot)$ of the previous paragraph simply projects $\scrX$ to the spatial coordinate in the case of the DARCM. For two vertices \(\x=(x,t_x)\) and \(\y=(y,t_y)\) with \(t_y<t_x\), we refer to \(\y\) being \emph{older} than \(\x\) and \(\x\) being younger than \(\y\), respectively. Almost surely, no vertices are born at the same time. Our choice of identifying the marks as birth times is rooted in the local limit construction of Section~\ref{sec:DPA}. To define the distribution of directed edges or \emph{arcs} in the graph we introduce the following parameters:
\begin{enumerate}[(i)]
	\item A \emph{spatial profile} \(\rho:(0,\infty)\to[0,1]\), which is non-increasing, satisfies $\int_0^\infty\rho(z)\d z<\infty$ and controls the spatial decay of connection probabilities;
	\item a parameter \(\gamma\in(0,1)\) to adjust the power-law exponent of the indegree distribution;
	\item an edge intensity \(\beta>0\);
	\item and a \emph{reciprocity profile} \(\pi: [1,\infty)\to[0,1]\), which is non-increasing, satisfies \(\pi(1)=1\), and parametrises the likelihood of mutual linkage.
\end{enumerate}     
The directed graph \(\D=\D(\beta,\gamma,\rho,\pi)\) is built using these parameters via the following procedure:
\begin{enumerate}[(A)]
	\item 
        Given \(\scrX\), each vertex \(\x=(x,t_x)\) forms an arc to each older vertex \(\y=(y,t_y)\) (i.e.\ \(t_y<t_x\)) independently of all other potential arcs with probability
		\begin{equation}\label{eq:firstConnect}
			\rho\big(\beta^{-1} t_y^\gamma t_x^{1-\gamma}|x-y|^d\big).		
		\end{equation}
		If a corresponding arc is formed, we denote this by \(\x\rightarrow\y\) or \(\y\leftarrow\x\).
	\item 
        Given \(\scrX\) and all arcs created in (A), each vertex \(\y=(y,t_y)\) sends a \emph{reverse arc} to each \(\x=(x,t_x)\) with \(\x\to \y\) independently of all other potential reverse arcs with probability 
		\begin{equation} \label{eq:backConnect}
			\pi\big(\nicefrac{t_x}{t_y}\big).
		\end{equation}
		If such a reciprocal connection is made, we denote this event by \(\x\leftrightarrow\y\).
\end{enumerate}
For the choice of \(\pi\equiv 1\) each arc always points in both directions. Therefore, the choice of \(\pi\equiv 1\) corresponds to constructing an undirected graph, which we denote by \(\G=\D(\beta,\gamma,\rho,1)\). The graph \(\G\) thus defined is known as \emph{age-dependent random connection model}~\cite{GGLM2019}. Clearly, \(\D\) can be constructed from \(\G\) by first pointing all edges in \(\G\) from younger end-vertex to older end-vertex and then adding reverse arcs by way of~\eqref{eq:backConnect}. Alternatively, we can also identify the model \(\D(\beta,\gamma,\rho,0)\) with a variant $\vec{\G}$ of the age-dependent random connection model in which all edges are directed from the younger to the older vertex.

Let us quickly explain the motivation behind our construction. The location of a vertex describes some intrinsic parameters; two vertices have a close affinity if they are spatially close to each other. The age of a vertex models directly its attractiveness in the graph, the older a vertex the more arcs it attracts. This is a simplified \emph{preferential attachment} mechanism, based on the observation that in true preferential attachment networks it is the old vertices that tend to accumulate a lot of arcs~\cite{JacobMoerters2015}. In social networks, users tend to `follow' a friend, i.e.\ someone with intrinsic affinity to them, or famous users who have already accumulated a lot of followers. This is built into our model: spatial closeness makes arcs likely and arcs to old and thus very attractive vertices are favoured. Both effects are reflected in the monotonicity of \(\rho\). The integrability condition, together with \(\gamma<1\), ensures that the expected number of arcs incident to any vertex remains finite. Further, it is easy to see that \(\beta\) controls the overall intensity of arcs. Once the arcs from younger to older vertices are formed, the latter may form an arc back to the younger vertex. In terms of social networks, the established user with many followers decides whether to follow some of their followers in return. This happens with probability \(\pi\). Since \(\pi(1)=1\) and \(\pi\) is monotone, if the two vertices are approximately of the same age, the occurrence of a reverse arc is quite probable. However, the older the old vertex is compared to the younger one the less likely the presence of the reverse arc becomes. 

Commonly used types of spatial profiles for the generation of spatial networks are either \emph{long-range profiles} of polynomial decay, which we parametrise as
\[
	\rho(x):=\rho_\delta(x)=1\wedge x^{-\delta}, \qquad \text{ for some } \delta>1,
\] 
or \emph{short-range profiles} which we identify with the indicator function \(\1_{[0,1]}\). Using the short-range profile in~\eqref{eq:firstConnect}, the younger vertex \(\x\) only sends an arc to the older vertex \(\y\) if \(t_y^\gamma t_x^{1-\gamma}|x-y|^d\leq \beta\). Hence, an arc is formed with probability one if the `age-scaled distance' \(t_y^\gamma t_x^{1-\gamma}|x-y|^d\) between them is at most \(\beta\). For long-range profiles arcs to vertices at arbitrarily large `age-scaled distance' are present with positive probability. The polynomial tail softens the geometric restrictions of the model: the smaller \(\delta>1\), the softer these restrictions are. In this article, we mostly work in the long-range regime. Results for the short-range profile can then often be derived by taking the limit \(\delta\to\infty\). We therefore also include the short-range profile into our standard parametrisation of $\rho$ by setting \(\rho_\infty=\1_{[0,1]}\). Similarly, we shall assume that the reciprocity profile is of the form \(\pi(t)=t^{-\Gamma}\) for \(\Gamma\geq 0\). Note that for \(\Gamma=0\), we have \(\pi\equiv 1\) yielding the original age-dependent random connection model as outlined above. With the profiles fixed, we have a parametrisation \(\D=\D[\beta,\gamma,\delta,\Gamma]\) in terms of the four real numbers \(\beta>0\), \(\gamma\in(0,1)\), \(\delta>1\), and \(\Gamma\geq 0\). We present our main results for the DARCM in the following section in terms of this parametrisation. This particularly includes the distribution of in and out-neighbourhoods as well as its weak percolation behaviour. 

\paragraph{Notation.} Throughout the manuscript, we use standard Landau notation. For two non-negative functions \(f,g\), we additionally use \(f\asymp g\) to indicate \(f=\Theta(g)\) (where we make use of the latter notation still). To keep notation concise, we may write \(\x\in\D\) for a vertex \(\x\in\scrV(\D)\). Furthermore, when summing over pairs of vertices of the graph, i.e.\ \(\sum_{\x,\y\in\D}\), the vertices are always to be understood as \emph{distinct}. However note, that our model does not include any self-loops so that, typically, diagonal pairs \((\x,\x)\) cannot contribute to the sum.

%%%%%%%%%%%%%%%%%%%%%%%%%%%
%%% Neighbourhoods %%%%%%%%
%%%%%%%%%%%%%%%%%%%%%%%%%%%
\section{Degree distributions of the DARCM}\label{sec:degrees}
    In this section, we give the degree distributions for in- and outdegree in the digraph \(\D=\D[\beta,\gamma,\delta,\Gamma]\). To properly formulate our results, we have to work in the \emph{Palm version} of the model~\cite{LastPenrose2017}. That is, we assign to the origin \(o\) an independent and uniformly distributed birth time \(U_o\) and add the vertex \(\o=(o,U_o)\) to \(\D\) by way of the above procedure described in~\eqref{eq:firstConnect} and~\eqref{eq:backConnect}. We denote the resulting graph by \(\D_o\) and the underlying probability measure by \(\P_o\). As the distinguished vertex \(\o\) can be seen as a typical vertex that has been shifted to the origin, we also refer to \(\o\) as the root vertex of \(\D_o\). 

For a given vertex \(\x\), let us denote by 
    \[
        \scrN^{\text{in}}(\x):=\{\y\in\scrX: \y\rightarrow\x\}
    \]
	the fan in of \(\x\) in \(\mathscr{D}\) and by \(\sharp\scrN^{\text{in}}(\x)\) its indegree. If \(\x=\0\) is the origin, we simply write \(\scrN^{\text{in}}\). For the fan out and the outdegree, we use the analogous notations \(\scrN^{\text{out}}(\x)\) or \(\scrN^{\text{out}}\), and \(\sharp\scrN^{\text{out}}(\x)\) or \(\sharp\scrN^{\text{out}}\), respectively. 

Let the vertex \(\x=(x,u)\) be given. Then by the same arguments as in~\cite[Prop.\ 4.1]{GGLM2019}, the in- and out-neighbourhood of \(\x\) form Poisson processes on \(\R^d\times(0,1)\) with respective intensity measures
\[
    \lambda_{\x}^\text{in}(\d(y,s)) = \Big(\1_{\{u\geq s\}}\big(\tfrac{s}{u}\big)^\Gamma\rho(\beta^{-1}u^{1-\gamma}s^\gamma|y-x|^d) + \1_{\{u<s\}}\rho(\beta^{-1} u^\gamma s^{1-\gamma}|y-x|^d) \Big)\d y \, \d s,
\]
and
\[
     \lambda_{\x}^\text{out}(\d (y,s)) = \Big(\1_{\{u\geq s\}}\rho(\beta^{-1}u^{1-\gamma}s^\gamma|y|^d) + \1_{\{u<s\}}\big(\tfrac{u}{s}\big)^\Gamma\rho(\beta^{-1} u^\gamma s^{1-\gamma}|y|^d) \Big)\d y \, \d s.
\]
We do not give the proof here but focus on the in and outdegree distribution, for which, however, this statement is of good use.
    
\begin{theorem}[Neighbourhoods in DARCM]\label{thm:neighbourhood}
    Let \(\beta>0\), \(\gamma\in(0,1)\), \(\delta>1\), and \(\Gamma\geq 0\). Almost surely, the root \(\0\) in \(\D_o\) has finite degree. More precisely, for \(k\in\N_0\), we have 
    \begin{enumerate}[(i)]
        \item 
    	   for the indegree
    	   \begin{equation*}
    		  \begin{aligned}
    		      & \P_o\big(\sharp\scrN^{\text{in}}=k\big)\asymp k^{-1-\nicefrac{1}{\gamma}}. 
    		  \end{aligned}	
    		\end{equation*}
    	\item 
            For the outdegree, we have, 
            \begin{enumerate}[(a)]
            	\item
            		if \(\Gamma>\gamma\), then \(\sharp\scrN^{\text{out}}\) is Poisson distributed with parameter \(\lambda\asymp \tfrac{\omega_d \delta \beta}{\delta-1}(\tfrac{1}{(\Gamma-\gamma)}+\tfrac{1}{\gamma})\), where \(\omega_d\) denotes the volume of the \(d\)-dimensional unit ball.
            	\item 
            		If \(\Gamma=\gamma\), then
            			\[
            		        \P_o\big(\sharp\scrN^{\text{out}}=k\big) \asymp \int_0^1 \frac{u \log(1/u)^k}{k!} \d u = 2^{-(k+1)}.
            			\]
            	\item 
            		If \(\Gamma<\gamma\), then 
            		\[
                		\begin{aligned}
               	    			& \P_o\big(\sharp\scrN^{\text{out}}=k\big)\asymp k^{-1-\nicefrac{1}{(\gamma-\Gamma)}}.
               	 		\end{aligned}
            		\]
            \end{enumerate} 		
    \end{enumerate}
\end{theorem}
	
By the refined Campbell formula~\cite[Theorem~9.1]{LastPenrose2017}, an immediate consequence of the above theorem is that the graph \(\D\) is locally finite. That is, each vertex has almost surely finite in- and outdegree, which are distributed according to the above distributions. We infer from the theorem that the indegree always follows a power law. This coincides with our understanding of social media networks where there are a few but noticeably many `influencers' with far more followers than the average. Whether the outdegree is heavy-tailed as well on the other hand depends on the strength of the reciprocity parameter \(\Gamma\). Note, however, that Part~(ii) implies that the number of original out-neighbours (i.e.\ older out-neighbours) is only Poisson distributed. As the additional out-neighbours are all also in-neighbours, the overall degree distribution of the underlying graph has power-law exponent \(1+\nicefrac{1}{\gamma}\), as shown in~\cite{GGLM2019}. This property is often referred to as \emph{scale-free}~\cite{BA99,AlbertBarabasi2002} to emphasise that there is no fast concentration of degrees around a `standard value' (the scale). However, as the notion of scale is quite misleading in a spatial setting, we shall simply refer to heavy-tailed or power-law degree distributions. For a more involved discussion of how the notion of scale may be interpreted in the degree context, we refer the reader to~\cite{ChungLu2006}.
	
%%%%%%%%%%%%%%%%%%%%%%%
%% Weak percolation %%%
%%%%%%%%%%%%%%%%%%%%%%%
\section{Out-percolation in the DARCM} \label{sec:perc}
Originally introduced by Broadbent and Hammersley in 1957~\cite{Broadbent1957}, percolation has drawn a lot of attention from the mathematical community and is widely studied until today. Percolation models use connectivity in random graph as a simple approximation for spreading or permeation dynamics. The original and most studied percolation model is \emph{Bernoulli bond percolation} on the Euclidean lattice \(\Z^d\). In this model, every edge of the nearest-neighbour graph on \(\Z^d\) is present with probability \(p\in(0,1]\), independently of all other edges, and absent with probability \(1-p\). The result of interest is then the existence of a critical parameter \(p_c\in(0,1)\) such that there is no infinite connected component for \(p<p_c\) but an infinite connected component exists almost surely for \(p>p_c\). Similarly, in \emph{Bernoulli site percolation} each vertex is present with probability \(p\) or removed with all its adjacent edges with probability \(1-p\). We refer to the book of Grimmet~\cite{Grimmett1999} for an overview on important results. As an equivalent in continuum space, Gilbert introduced in 1961 a model where the vertices are given through a Poisson process and each pair of vertices is connected if their distance is smaller than some threshold \(\beta\)~\cite{Gilbert61}. In this situation, one is interested in the existence of a critical threshold \(\beta_c\in(0,\infty)\) above which an infinite connected component is present and below which the graph decomposes into finite components. By the thinning property of the Poisson process, it is either possible to first construct the graph and then remove each vertex independently with probability \(p\), or to vary the Poisson intensity before constructing the graph to obtain the same model. In particular, \(\beta\) in Poissonian based Gilbert graphs and in the DARCM can be interpreted either as scaling distances, as varying the Poisson intensity, or as parametrising site-percolation on the resulting graph~\cite{LastPenrose2017}. For an overview on continuum percolation, we refer to the book of Meester and Roy~\cite{MeesterRoy1996}. Many more percolation models have been studied since, both on lattices~\cite{Schulman1983,Yukich2006,DeijfenHofstadHooghiemstra2013}, and in the continuum~\cite{DeprezWuthrich2019,GGLM2019,GLM2021,BodeFountoulakisMuller2015}. 

In the context of networks, percolation may serve to model global breakdown if each vertex has an independent local failure probability \(1-p\). Here, a network breakdown is to be understood as a decomposition from a  basically connected network to only finite components, effectively preventing non-local communication. Therefore, percolation concepts play an important role in our understanding of network effects. Another interpretation is the following: consider, for example, an infinite social media network in which a message is shared. Each vertex that receives the message shares it with all its neighbours with probability \(p\). If \(p>p_c\), then the message may spread through the network for eternity, while for \(p<p_c\) the message is guaranteed to only reach a finite number of vertices. In particular, if \(p_c=0\) or equivalently \(\beta_c=0\), then a message may spread through the network forever, even if local random defects are very likely. 

In our directed setting, there are now three types of connectedness. Let us denote by \(\x \rightsquigarrow \y\) the event that there exists a path from \(\x\) to \(\y\) using only outgoing arcs. Then, \(\y\) is said to be in the \emph{out-fan} of \(\x\) if \(\x\rightsquigarrow\y\) and \(\y\) is said to be in the \emph{in-fan} of \(\x\) if \(\y\rightsquigarrow\x\). Furthermore, \(\x\) and \(\y\) belong to a \emph{strongly connected component} if both \(\x\rightsquigarrow\y\) and \(\y\rightsquigarrow\x\) hold, which we denote from here on onwards as \(\x\leftrightsquigarrow\y\). Since, we work in the continuum, the natural parameter for our percolation question is the edge intensity \(\beta\) as outlined above. We define two critical intensities
\begin{equation*}%\label{eq:weakCriticalBeta}
	\begin{aligned}
		\overset{\rightarrow}{\beta}_c 
		& 
			:= \overset{\rightarrow}{\beta}_c(\gamma,\delta,\Gamma) 
            = 
			\sup\big\{\beta>0: \P\{\o\rightsquigarrow\infty\}=0\big\},
	\end{aligned}
\end{equation*}
where \(\o\rightsquigarrow\infty\) denotes the event that the origin starts an infinite long directed (self-avoiding) path of outgoing arcs. Similarly, we define 
\begin{equation*}%\label{eq:strongCriticalBeta}
	\begin{aligned}
	\overset{\leftrightarrow}{\beta}_c 
	&
		:= \overset{\leftrightarrow}{\beta}_c(\gamma,\delta,\Gamma) 
		= \sup\big\{\beta>0: \P\{\o\leftrightsquigarrow\infty\}=0\big\},
	\end{aligned}
\end{equation*}
where \(\o\leftrightsquigarrow\infty\) denotes that \(\o\) starts two infinite paths, one consisting of outgoing and one of ingoing arcs.

We are interested in whether the directed graph \(\D\) possesses a subcritical out-percolation phase, i.e.\ for which parameters is $$\overset{\leftrightarrow}{\beta}_c\geq \overset{\rightarrow}{\beta}_c>0,$$ such that for small enough \(\beta\) there is no undirected path to infinity. We focus on out-percolation, since the DARCM is mainly intended as a toy model for social network formation, where information would naturally spread in opposite arc direction and a directed connection to $\infty$ would correspond to the spread of global `viral' trends. Denoting 
\begin{equation*}
    \overset{\rightarrow}{\theta}(\beta) = \P_o(\0\rightsquigarrow\infty \text{ in }\D),
\end{equation*}
we obtain the following result:

\begin{theorem}[Existence vs.\ non existence of an out-percolation phase] ~\ \label{thm:perc}
\begin{enumerate}[(i)]
    \item If \(\gamma<\nicefrac{(\delta+\Gamma)}{(\delta+1)}\), then \(\overset{\rightarrow}{\theta}(\beta)=0\) for all sufficiently small \(\beta\) and consequently \(\overset{\rightarrow}{\beta}_c>0\).
    \item If \(\gamma>\nicefrac{(\delta+\Gamma)}{(\delta+1)}\), then for all \(\beta>0\), we have \(\overset{\rightarrow}{\theta}(\beta)>0\) and consequently \(\overset{\rightarrow}{\beta}_c =0\).
\end{enumerate} 
\end{theorem}

It would be interesting to derive a more complete picture of the percolation regimes by also considering in-percolation, strong percolation, and the regime boundary. We leave this for future work.

\section{Directed age-based preferential attachment} \label{sec:DPA}
We now discuss a directed version of the age-based preferential attachment model~\cite{GGLM2019} from which the DARCM arises naturally as local limit; recall Figure~\ref{fig:simulation} for a simulation. Choose \(\beta>0\), \(\gamma\in(0,1)\), \(\delta>1\), and \(\Gamma\geq 0\) to build a growing sequence of directed graphs \((\D_t:t\geq 0)\) as follows: At time \(t=0\), the graph \(\D_0\) is the empty graph consisting of neither vertices nor arcs. Then
	\begin{itemize}
		\item vertices arrive successively after independent exponential waiting times with mean \(1\) and are placed uniformly on the unit torus \([\nicefrac{1}{2},\nicefrac{1}{2})^d\).
		\item Given the graph \(\D_{t-}\), a vertex \(\x=(x,t)\) newly born at time \(t\) and placed at \(x\) forms an arc to each already existing vertex \(\y=(y,s)\), born at an earlier time \(s<t\) and located at \(y\), independently with probability 
		\begin{equation*}%\label{eq:PAconnect}
			\rho_\delta\Big(\frac{t \operatorname{d}_1(x,y)^d}{\beta (\nicefrac{t}{s})^\gamma}\Big),
		\end{equation*}
        where 
        \[
    	\operatorname{d}_1(x,y):=\min\big\{|x-y+u|:u\in\{-1,0,1\}^{\times d}\big\} 
        \]
		denotes the torus metric. If the arc \(\x\rightarrow\y\) has been formed, the older vertex \(\y\) forms an arc to \(\x\) independently with probability \((\nicefrac{s}{t})^\Gamma\).
	\end{itemize}
	
	The idea of preferential attachment goes back to Barab\'{a}si and Albert~\cite{BA99}, spatial versions of preferential attachment were introduced in \cite{AielloEtAl2008,JacobMoerters2015}. The age-dependent random connection model of \cite{GGLM2019} aims at approximating the model of \cite{JacobMoerters2015} whilst keeping as much independence in the connection mechanism as possible. The same approach motivates our directed version.

    We next explain how the process \((\D_t:t\geq 0)\) of finite digraphs has the random graph \(\D\) as its local limit. The notion of weak local limit was introduced independently by Benjamini und Schramm~\cite{BenjaminiSchramm2001}, and by Aldous and Steele~\cite{AldousSteele2004} to study local characteristic of growing graph sequences. Here, we use the version of \emph{local limit in probability}~\cite{vanderHofstadRGCN2_2024}, which is more suitable in the context of random graphs. 
    Loosely speaking, the local limit, which is always a rooted graph, approximates the local neighbourhood of a uniform chosen vertex in the sense that the subgraph induced by all vertices within fixed finite graph distance \(k\) looks with high probability like the subgraph up to graph distance \(k\) of the root in the limiting graph. While this concept is well-established in the undirected graph setting~\cite{vanderHofstadRGCN2_2024}, the undirected graph setting is often more involved. This is due to the fact that the exploration process inducing the subgraph of bounded graph distance is less clear because of the arcs' orientations~\cite[Remark~1.1]{vanderhofstad2024giants}. However, unlike in the general situation, we can make use of the structure provided by the underlying Poisson process. As we shall see in Section~\ref{sec:edge-marked}, our construction allows us to encode all randomness of \(\D_t\) and \(\D\) in terms of marked ergodic point processes. That is, the according point process contains all relevant information required to build the digraphs \(\D_t\) and \(\D\). This enables us to apply law of large numbers for point processes~\cite{PenroseYukich2003}, which ultimately give rise to even slightly stronger limit results in the same vein as in~\cite{JacobMoerters2015}, see Proposition~\ref{prop:LLN} below. Particularly, our local limit is stronger than the one established in~\cite{Garavaglia_2020_pagerank,vanderhofstad2024giants}  

    Let us call any digraph with a distinguished root vertex as a rooted digraph. Recall \(\D_o\), the DARCM with a root vertex added at the origin. 
                
    \begin{theorem}[DARCM as local limit]\label{thm:DARCMasLL}
    	Let \(H\) be a non-negative functional that acts on rooted digraphs and their root and that only depends on a bounded graph neighbourhood of the root in the undirected sense. Assume further that the family
    	\begin{equation*} 
    		\Big(\frac{1}{\sharp \scrV(\D_t)}\sum_{\x\in \D_t}H(\x,\D_t)\Big)_{t\geq 0}
    	\end{equation*} 
    	is uniformly integrable. Then, we have in probability and in \(L^1\),
		\begin{equation}\label{eq:DARCMasLL}
			\frac{1}{\sharp \scrV(\D_t)}\sum_{\x\in \D_t}H(\x,\D_t) \longrightarrow \E_o[H(\0,\D_o)],
		\end{equation}
		as \(t\to\infty\). 
    \end{theorem}
    \begin{remark} ~\
    	\begin{enumerate}[(i)]
    		\item 
    			It is important to note that \(H\) solely depends on the digraph structure but does not on the vertices' embedding or birth times. In fact, the law of large numbers, Proposition~\ref{prop:LLN} below, will allow for additional dependences on these things.
    		\item
    			Functionals in this settings are typically not allowed to depend on the labelling of the digraph and must remain unchanged under isomorphisms. More precisely, they act on the residue class of the digraph modulo graph isomorphism. However, as the underlying spatial embedding provides a unique labelling of the vertices via their location, this problem does not occur in our setting.  
    	\end{enumerate}	
    \end{remark}
    
    \subsection{Empirical degree distribution} 
    As an immediate consequence of the local limit theorem and the results of Section~\ref{sec:degrees}, we obtain bounds for the empirical in- and outdegree distributions.
    
    \begin{theorem}
    	Consider the family of digraphs \((\D_t\colon t>0)\).
    	\begin{enumerate}[(i)]
    		\item 
    			As \(t\to\infty\), we have in probability, for each \(k\in\N_0\),
    			\[
    				\frac{\sharp\{\x\in\scrV(\D_t)\colon \sharp\scrN_t^{\text{in}}(\x)=k\}}{\sharp\scrV(\D_t)} \longrightarrow \P_o\{\sharp\scrN^{\text{in}}=k\} \asymp k^{-1-\nicefrac{1}{\gamma}},
    			\]
    			where \(\sharp\scrN_t^{\text{in}}(\x)=k\) denotes the indegree of vertex \(\x\) in \(\D_t\).
    		\item	
    			Similarly, as \(t\to\infty\), we have for the empirical outdegree, in probability for each \(k\in\N_0\),
    			\[
    				\frac{\sharp\{\x\in\D_t\colon \sharp\scrN_t^{\text{out}}(\x)=k\}}{\sharp\scrV(\D_t)} \longrightarrow \P_o\{\sharp\scrN^{\text{out}}=k\}
                    \asymp
    				\begin{cases}
    					e^{-\beta/(\Gamma-\gamma)} \tfrac{(\beta/(\Gamma-\gamma))^k}{k!}, & \text{ if } \Gamma>\gamma, 
    					\\
    					2^{-k-1}, & \text{ if } \Gamma=\gamma,
    					\\
    					k^{-1-1/(\gamma-\Gamma)}, & \text{ if } \Gamma<\gamma.
    				\end{cases}
    			\]
    	\end{enumerate}
    \end{theorem}
    
    Another consequence of Theorem~\ref{thm:DARCMasLL} is the sparsity of the digraph family \((\D_t)_t\), that is, the number of arcs is proportional to the number of vertices. 
    
    \begin{theorem} \label{thm:sparse}
    	The family \((\D_t\colon t>0)\) is sparse, i.e.
    	\[
    		\frac{\sharp\{\text{arcs in }\D_t\}}{\sharp\scrV(\D_t)}\longrightarrow \E_o\big[\sharp\scrN^{\text{out}}\big]\in(0,\infty),
    	\]
    	in probability, as \(t\to\infty\).
    \end{theorem}    
    
    %%%%%%%%%%%%%%%%%%%%%%%%%
    %%%% CLustering %%%%%%%%%
    %%%%%%%%%%%%%%%%%%%%%%%%%
    \section{Clustering in the DARCM and the directed age-based PA model} \label{sec:cluster}
    Evidence suggests that vertices in social media networks tend to form rather dense local subgraphs, often referred to as \emph{filter bubbles}. These bubbles are typically either centred around a common group of friends or based on similar interests. A tool often used to measure the tendency towards these filter bubbles within the networks are \emph{clustering coefficients}. In this section, we discuss two such coefficients, one based on friendship and one based on shared interests. Essentially one measures how likely two vertices with a common friend are friends themselves, and how likely it is that two users that have one common interest share another. Here, `friendship' denotes a reciprocal connection and an `interest' is an outgoing arc, see Figure~\ref{fig:clustering}. Let us start with the friendship clustering as it is closely related to the standard triangle count in undirected networks.
    
    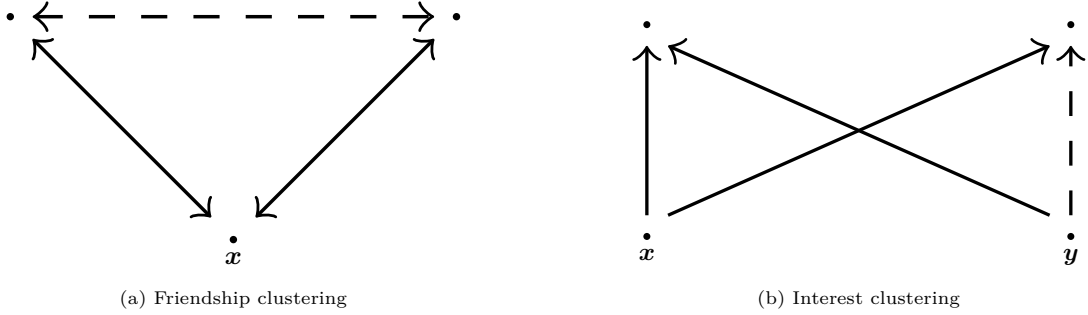
\begin{figure}
    	\begin{center}
    		\begin{subfigure}[t]{0.45\textwidth}
    			\resizebox{\textwidth}{!}{
					\begin{tikzpicture}[scale=0.3, every node/.style={scale=0.3}]
						\draw (-3 , 3) node[circle, fill = black, scale=0.3] (y) {};
						\draw (0 , 0) node[circle, fill = black, scale=0.3, label=below:{$\x$}] (x) {};
						\draw (3 , 3) node[circle, fill = black, scale=0.3] (z) {};
						\draw[<->, thin] (-0.3, 0.3) -- (-2.7, 2.7);
						\draw[<->, thin] (0.3, 0.3) -- (2.7, 2.7);
						\draw[<->, thin, dashed] (-2.7, 3) -- (2.7, 3);
					\end{tikzpicture}
				}
				\caption{Friendship clustering}
				\label{fig:friend}
    		\end{subfigure}
    		\hfill
    		\begin{subfigure}[t]{0.45\textwidth}
    			\resizebox{\textwidth}{!}{
					\begin{tikzpicture}[scale=0.3, every node/.style={scale=0.3}]
						\draw (-3 , 3) node[circle, fill = black, scale=0.3] (v) {};
						\draw (-3 , 0) node[circle, fill = black, scale=0.3, label=below:{$\x$}] (x) {};
						\draw (3 , 0) node[circle, fill = black, scale=0.3, label=below:{$\y$}] (y) {};
						\draw (3 , 3) node[circle, fill = black, scale=0.3] (z) {};
						\draw [->, thin] (-3, 0.3) -- (-3, 2.7);
						\draw [->, thin] (-2.7, 0.3) -- (2.7, 2.7);
						\draw [->, thin] (2.7, 0.3) -- (-2.7, 2.7);
						\draw [->, thin, dashed] (3, 0.3) -- (3, 2.7);
					\end{tikzpicture}
				}
				\caption{Interest clustering}
				\label{fig:interest}
    		\end{subfigure}
    	\end{center}
    	\caption{Depiction of the two clustering metrics. The labeled vertices \(\x\) and \(\y\) refer to the local versions.}
    	\label{fig:clustering}
    \end{figure}
    
    \paragraph{Friendship clustering.}
    Let us call two given vertices \(\x\) and \(\y\) \emph{friends} if \(\x\leftrightarrow\y\). In a nutshell, friendship clustering measures how likely nodes with a common `friend' are friends themselves. This can be done from a local or a global perspective. In the local perspective, two vertices are sampled from the friend neighbourhood of a typical vertex and one asks whether these vertices are friends themselves. In the global perspective one considers the probability that a uniformly chosen open triangle, formed by bidirectional arcs, is closed, see Figure~\ref{fig:friend}. Since this coefficient is build on triangles, it is closely related to the standard clustering coefficient in undirected graphs and our case is essentially a straightforward adaption of the setting in~\cite{GGLM2019}. 
    
    Let us start with the local viewpoint and let \({\scrV_t^{(2)}}\) the set of all vertices having at least two friends in \(\D_t\). For \(\x\in\scrV_t^{(2)}\), we define the \emph{local friend clustering coefficient} of \(\x\) in \(\D_t\) as 
\[
    c^\textup{fc}(\x, \D_t) = \frac{\sum\limits_{\substack{\y,\z\in\D_t: \, t_y>t_z}}\1_{\{\x\leftrightarrow\y\}}\1_{\{\x\leftrightarrow\z\}}\1_{\{\y\leftrightarrow\z\}}}{\binom{\sharp (\mathscr{N}^\text{out}(\x)\cap\mathscr{N}^\text{in}(\x))}{2}}.
\]
If \(\x\not\in\scrV_t^{(2)}\), we simply set its friend clustering coefficient to be zero.
To make use of the local limit structure and make certain that the considered vertices are sampled from a typical friend neighbourhood, typically the \emph{average friend clustering coefficient} is considered, which is defined as
\[
	c^\textup{fc}_\textup{av}(\D_t)  = \frac{1}{\scrV_t^{(2)}} \sum_{\x\in\D_t}c^\textup{fc}(\x,\D_t).
\]

\begin{theorem}[Average friend clustering] \label{thm:locFriendClust} 
    For all \(\beta>0,\gamma\in(0,1),\delta>1,\) and \(\Gamma\geq 0\), we have, in probability as \(t\to\infty\),
    \[
    	c^\textup{fc}_\textup{av}(\D_t) \longrightarrow \E_o c^\textup{fc}(\o,\D_o),
    \]
    where
    \begin{equation*}
        \begin{aligned}
            \E_o c^\textup{fc}(\0, \D_o) & = %\int_{0}^1\ \P(\Y^{(u)}\leftrightarrow\X^{(u)})\P_{(0,u)}\Big(\bigcup_{k\geq 2}F_{(0,u)}(k)\Big)\d u >0,
            \int_0^1 \P(\Y^{(u)}\leftrightarrow\X^{(u)}) \mu^f_2(\d u),
        \end{aligned}
    \end{equation*}
    %and \(F_{(o,u}(k)\) is the event that the root \((o,u)\) has \(k\) friends while 
    and \(\X^{(u)}\) and \(\Y^{(u)}\) are two independent random variables distributed according to the normalised measure \(\lambda_u^f/\lambda_u^f(\R^d\times(0,1))\) with
        \[
       \lambda_{u}^\textup{f}(\d(x,s)) =  \Big((\tfrac{s}{u})^\Gamma\rho\big(\beta^{-1} s^{\gamma}u^{1-\gamma}|x|^d\big)\1_{\{s<u\}}+(\tfrac{u}{s})^\Gamma\rho\big(\beta^{-1}s^{1-\gamma}u^\gamma |x|^d\big)\1_{\{s\geq u\}}\Big) \d s \, \d x,
    \]
    and \(\mu^f_2\) is the distribution on \((0,1)\) with density proportional to 
    \[
        1-e^{-\lambda_u^f(\R^d\times(0,1))}-\lambda_u^f(\R^d\times(0,1))e^{-\lambda_u^f(\R^d\times(0,1))}.
    \]
\end{theorem}

The global friend clustering coefficient is now defined as the proportion of open bidirectional triangles (i.e.\ one double-arc may be missing) that are closed (i.e.\ in fact all three double-arcs are present). Thus,
\[
	c^\textup{fc}_\text{glob}(\D_t) = 3\frac{\sum_{\x,\y,\z\in\cD_t} \1_{\{\x\leftrightarrow\y\}}\1_{\{\x\leftrightarrow\z\}}\1_{\{\y\leftrightarrow\z\}}}{\sum_{\x,\y,\z\in\D_t} \1_{\{\x\leftrightarrow\y\}}\1_{\{\x\leftrightarrow\z\}}},
\]
where the factor \(3\) takes into account that each closed triangle creates three open triangles present in the denominator. 

\begin{theorem}[Global friend clustering] \label{thm:globFriendClust}
	 For all \(\beta>0,\gamma\in(0,1),\delta>1,\) and \(\Gamma\geq 0\), there exists a constant \(c\geq 0\) such that, in probability as \(t\to\infty\),
	 \[
	 	c^\textup{fc}_\textup{glob}(\D_t) \longrightarrow c,
	 \]
	where \(c>0\) if and only if \(\gamma-\Gamma<1/2\). 
\end{theorem}

\newpage

\begin{remark}~\
    \begin{enumerate}[(i)]
        \item
        	The distribution \(\mu_2^f\) appearing in the representation of the limiting average clustering coefficient can be interpreted as the distribution that chooses a typical vertex among all vertices with at least two friends. Note that the average clustering coefficient is always positive.
        \item   
            The global friend clustering coefficient may be positive or zero since it is not localised at typical vertices but considers all closed and open triangles. This may put considerably more mass on old vertices with high degrees. If \(\gamma-\Gamma\geq 1/2\), then old vertices collect so many arcs \emph{and reconnect} to so many of these that they produce extraordinary many open triangles, most of which are not closed. Note that, if \(\Gamma\geq \gamma\), we always have \(\gamma-\Gamma\leq 0<1/2\) and the coefficient is positive. In that case, it does not matter how many in-neighbours the old vertices collect as they only form few reciprocal arcs and only those double arcs are considered by the coefficient.
    \end{enumerate}
\end{remark}

\paragraph{Interest clustering.} 
While the friendship clustering coefficient is build on triangles made of double arcs, the interest clustering coefficient is based on directed `bow-ties', see Figure~\ref{fig:interest}, and basically measures how likely it is that two vertices with a common interest also share another. In graph theory notion, a bow-tie is the directed bipartite subgraph \(K_{2,2}\) and we are again interested in the ration of closed vs.\ open bow-ties, i.e.\ the subgraph \(K_{2,2}\) with at most one arc missing. This coefficient was studied under the name \emph{diclique clustering coefficient} in~\cite{Bloznelis_2016_diclique} for a directed random graph that is build from a bipartite graph of nodes and attributes. The notion of \emph{interest clustering coefficient} goes back to~\cite{TrollietEtAl2022}. In their work, the authors apply various clustering coefficients to the Twitter network and find that the interest clustering coefficient is particularly accurate in the directed social network setting as it combines the social with the informational aspect of filter bubbles. 

Again, we propose a local and a global viewpoint on the interest clustering coefficient and start with the local perspective. The coefficient can be described as the probability that two nodes that follow a common interest also share another node that they follow. While we formulated the local friendship coefficient by point of view of a single vertex, the interest clustering coefficient is formulated in terms of pairs of vertices.

Let \(\scrI_t\) be the set of all pairs of vertices in \(\D_t\) in which the first vertex has outdegree at least two and shares at least one out-neighbour with the second vertex. That is,
\[
	\scrI_t = \Big\{(\x,\y)\in \D_t\times \D_t\colon \x\neq \y, \sharp\scrN_t^{\text{out}}(\x)\geq 2,  \sharp\big(\scrN_t^\text{out}(\x)\cap \scrN_t^\text{out}(\y)\big)\geq 1,\Big\}.
\] 
For \((\x,\y)\in\scrI_t\), we define the \emph{local interest clustering coefficient} of \((\x,\y)\) in \(\D_t\) as
\[
	c^\text{ic}((\x,\y),\D_t) = 2\frac{\sum_{\u,\v\in\D_t} \1_{\{\y\rightarrow \u,\y\rightarrow \v\}}\1_{\{\x\rightarrow \u, \x\rightarrow \v\}}}{\sum_{\u,\v\in\D_t} \1_{\{\y\rightarrow \u \text{ or } \y\rightarrow \v\}}\1_{\{\x\rightarrow \u, \x\rightarrow \v\}}},
\] 
 while for \((\x,\y)\not\in\scrI_t\), we simply set \(c^\text{ic}((\x,\y),\D_t) =0\).
The corresponding \emph{average interest clustering coefficient} is then defined as
\[
	c^\text{ic}_\text{av}(\D_t) = \frac{1}{\sharp\scrI_t}\sum_{(\x,\y)\in\scrI_t} c^\text{ic}((\x,\y),\D_t).
\]

\begin{theorem}[Average interest clustering] \label{thm:locIC}
	For all \(\beta>0\), \(\delta>1\), and \(\Gamma\geq 0\), we have, in probability as \(t\to\infty\),
	\begin{enumerate}[(i)]
		\item 
			for \(\gamma<1/2\) that
			\[
				c_\textup{av}^\textup{ic}(\D_t) \longrightarrow \frac{\E_o\big[\sum_{\y\in\D_o} c^{\textup{ic}}((\o,\y),\D_o) \1_{\{\scrN^\text{out}\cap \scrN^\text{out}(\y)\neq \emptyset\}}\, \big| \, \scrN^\text{out}\geq 2\big]}{\E_o\big[\sharp\{\y\colon \scrN^\text{out}\cap \scrN^\text{out}(\y)\neq \emptyset\} \, \big| \, \scrN^\text{out}\geq 2\big]},
			\]
		\item
			and for \(\gamma\geq 1/2\), that \(c_\textup{av}^\textup{ic}(\D_t)\to 0\).
	\end{enumerate}
\end{theorem}
\begin{remark}[Size-biasing in second-order normalisations]
The interest clustering coefficient is normalised by counts of length-two configurations (pairs of out-neighbours and shared out-neighbours). In sparse heavy-tailed networks, such second-order normalisations typically induce a size-biasing effect: sampling along two-step structures disproportionately favours vertices with unusually large (out- or in-)neighbourhoods. This is analogous to the friendship paradox \cite{feld_91_fiendhsip,Hazra_f2025_friendship} and implies that the asymptotic behaviour of interest-type clustering statistics is governed not only by local motif probabilities but also by the finiteness of suitable second moments of degree counts. In our setting this mechanism leads to the threshold $\gamma=1/2$ that appears in Theorem~\ref{thm:locIC}. A similar effect can be observed by local triangle counts, e.g.\ the average friendship clustering, when the open triangles are not considered from their tip but from a boundary vertex. That is, instead of pairs of neighbours, the denominator counts the number of paths of length two. This metric is then also referred to as \emph{closure coefficient}~\cite{Yin19_closure, vdH_2018_closure, Stegehuis_2020_closure}.
\end{remark}

As in the case of friendship clustering, we also provide a global variant of the interest clustering coefficient. We define the \emph{global interest clustering coefficient} as 
\[
    c_\text{glob}^\text{ic}(\D_t)=2\frac{\sum_{\x,\y,\u,\v\in\D_t}\1_{\{\y\rightarrow \u,\y\rightarrow \v\}}\1_{\{\x\rightarrow \u, \x\rightarrow \v\}}}{\sum_{\x,\y,\u,\v\in\D_t} \1_{\{\y\rightarrow \u \text{ or } \y\rightarrow \v\}}\1_{\{\x\rightarrow \u, \x\rightarrow \v\}}}.
\]

\begin{theorem}[Global interest clustering]\label{thm:globIC}
    For all \(\beta>0,\gamma\in(0,1),\delta>1,\) and \(\Gamma\geq 0\), there exists a constant \(c\geq 0\) such that, in probability as \(t\to\infty\),
	 \[
	 	c^\textup{ic}_\textup{glob}(\D_t) \longrightarrow c,
	 \]
	where \(c>0\) if and only if \(\gamma<1/2\).
\end{theorem}

%%%%%%%%%%%%%%%%%%%%%%%%
%%%%%%% Zusammenfassung/Ausblick %%%%%%%%%
%%%%%%%%%%%%%%%%%%%%%%%%  
\section{Conclusion and future work}\label{sec:outlook}
We introduced and analysed a directed spatial network model that augments the age-dependent random connection model by an explicit reciprocity mechanism. The model remains analytically tractable through its Poissonian embedding while capturing empirically prominent directed features such as heterogeneous reciprocation and non-trivial directed clustering. Our main results identify sharp regimes for in- and outdegree behaviour, provide laws of large numbers for empirical network statistics through a marked-point-process representation, and establish conditions under which local and global clustering coefficients remain strictly positive in the sparse limit. In addition, we proved a dichotomy for out-percolation, separating parameter regions with a non-trivial critical intensity from regimes in which out-percolation occurs for arbitrarily small intensities.

Several natural extensions remain open. First, it would be of interest to quantify \emph{graph distances} and the distribution of \emph{edge lengths} (in both the Euclidean and age-scaled metrics) and to relate these to navigation and reachability properties. Secondly, beyond out-percolation, one may study in-percolation, strong percolation, and the structure and robustness of large weakly/strongly connected components, including potential ``weak giant'' phenomena in the sense of directed local exploration. Thirdly, the law of large numbers for vertex-edge markings allows in principle for more global functionals than those treated here, and it would be worthwhile to develop a systematic toolbox for such applications in directed settings. Finally, from a data-science viewpoint, the model suggests concrete directions for \emph{statistical fitting} (e.g.\ via degree and motif statistics) and for the analysis of ranking measures such as \emph{PageRank} in the presence of tunable reciprocity.

%%%%%%%%%%%%%%%%%%%%%%%%
%%%%%%% Proofs %%%%%%%%%
%%%%%%%%%%%%%%%%%%%%%%%%  
\section{Proofs of main theorems} \label{sec:proofs}
In this section we provide proofs for our results. We start with a graphical construction of the model via \emph{independent vertex-edge markings} for both the infinite model \(\D\) as well as the family \((\D_t\colon t\geq 0)\).

\subsection{Construction from an vertex-edge-marked point process}\label{sec:edge-marked}
We now construct the DARCM as a deterministic map of a Poisson point process together with vertex and edge marks. To this end, let \(\eta\) be a unit intensity Poisson point process on \(\R^d\). We may enumerate \(\eta=(X_1,X_2,\dots)\), cf.~\cite{LastPenrose2017}, and call the elements of \(\eta\) the \emph{vertex locations}. Let further \(\mathscr{T}=(T_i:i\in\N)\) be an i.i.d.\ sequence of random variables distributed uniformly on \((0,1)\) which is independent of \(\eta\). The marked Poisson process \(\scrX\) can then be represented as
\[
	\scrX=\big(\X_i=(X_i,T_i)\in\eta\times\mathscr{T}:i\in\N\big).
\]
Additionally, let \((U_{i,j}:i<j\in\N)\) be another i.i.d.\ sequence of \(\operatorname{Uniform}(0,1)\) random variables independent of \(\scrX\). We define \(\mathscr{U}=(U_{i,j}:i,j\in\N)\) whose elements we call \emph{edge marks} by setting \(U_{j,i}=U_{i,j}\) for \(i<j\) and \(U_{i,i}=0\). We then set
\begin{equation*}%\label{eq:edgeMarking}
	\xi=\big(((\X_i,\X_j),U_{i,j})\in\scrX^{2}\times\mathscr{U}:i,j\in\N\big)
\end{equation*}
and call \(\xi\) an \emph{independent vertex-edge marking} of $\scrX$ in accordance with the construction in~\cite{HvdHLM20}. Observe that the precise construction of \(\xi\) may depend on the ordering of \(\eta\) but its law does not. Clearly, \(\scrX\), \(\eta\), and \(\mathscr{T}\) can be recovered from \(\xi\). 

We now fix parameters \(\beta>0\), \(\gamma\in(0,1)\), \(\delta>1\), and \(\Gamma\geq 0\) and construct a digraph \(\D^{\beta,\gamma,\delta,\Gamma}(\xi)\) with vertex set \(\scrX\) and arc set
\begin{equation}\label{eq:edgeSet}
	\Big\{(\X_i,\X_j):\; U_{i,j}< \1_{\{T_i<T_j\}}\Big(\tfrac{\beta}{T_i^\gamma T_j^{1-\gamma}|X_i-X_j|^d}\Big)^{\delta} + \1_{\{T_i\geq T_j\}}\big(\tfrac{T_j}{T_i}\big)^\Gamma\big(\tfrac{\beta}{T_i^\gamma T_j^{1-\gamma}|X_i-X_j|^d}\big)^{\delta}\Big\},
\end{equation}
where the ordered tuple \((\X_i,\X_j)\) represents the arc \(\X_i\rightarrow\X_j\). Observe that \(\D^{\beta,\gamma,\delta,\Gamma}(\xi)\) and \(\D[\beta,\gamma,\delta,\Gamma]\) have the same law. We work from here on onwards on the full measurable space on which \(\xi\) is given, together with a probability measure \(\P\) with associated expectation \(\E\). Thus, we may identify \(\D=\D[\beta,\gamma,\delta,\Gamma]\) with \(\D^{\beta,\gamma,\delta,\Gamma}(\xi)\) and continue using the notation \(\D\). For the underlying undirected graph, we use the analogous representation \(\G = \D^{\beta,\gamma,\delta,0}(\xi)\). Let us remark that~\eqref{eq:edgeSet} shows that our model can be viewed as a directed \emph{weight-dependent random connection model} \cite{GHMM2022,GLM2021}, i.e.\ such a model with a non-symmetric connection function.

To add a root to our graph, we use the Palm version~\cite{LastPenrose2017} of \(\xi\). To this end, denote by \(\x_0=(0,t_0)\) a vertex located at the origin with given birth time \(t_0\). Let us denote by \(\scrX_{\x_0}=\scrX\cup\{(x_0,t_0)\}\) the vertex set with the additional vertex at the origin. We further extend the edge-mark collection by the sequence of independent uniform random variables \((U_{0,j}:j\in\N)\) (together with \(U_{0,0}=0\) and \(U_{j,0}=U_{0,j}\)) and denote the resulting sequence by \(\mathscr{U}_{\x_0}\). Finally, we define the edge marked process containing the origin as
	\[
		\xi_{\x_0}:=\xi \cup\big\{\big((\x_0,\X_i),U_{0,i}\big):i\in\N\big\}\cup \big\{\big((\X_i,\x_0),U_{0,i}\big):i\in\N\cup\{0\}\big\}.	
	\]
	The graph containing a vertex located at the origin is then constructed in the obvious manner as \(\D_{\x_0}=\D^{\beta,\gamma,\delta,\Gamma}(\xi_{\x_0})\). We denote the probability measure and expectation governing \(\xi_{\x_0}\) by \(\P_{\x_0}\) or \(\P_{(x_0,t_0)}\) and \(\E_{\x_0}\) respectively. If the vertex mark of \(\x_0\) is yet unrevealed and thus uniformly distributed, we denote the root vertex by \(\0=(o,T_o)\) and set \(\P_{o}:=\int_0^1\P_{(o,t)}\d t\). The distribution \(\P_o\) corresponds to \(\xi\) (and thus \(\D\)) seen from the location of a typical vertex ~\cite[Chapter~9]{LastPenrose2017}. 
	
	In the same way, finitely-many given vertices \(\y_1=(y_1,s_1), \y_2=(y_2,s_2),\dots\) may be added to the graph using negative indices and writing \(y_i=x_{-i}\) and \(s_i=t_{-i}\). We then write \(\xi_{\y_1,\y_2,\dots}\), and \(\P_{\y_1,\y_2,\dots}\), etc. Note again that the above construction depends on the labelling of the vertices but the law does not. In the following the indices used play no more particular role and we refer to vertices as \(\x=(x,t_x)\) or \(\y=(y,t_y)\) just like in the introduction. If needed, we index the edge mark by the corresponding per of vertices, e.g.\ \(U_{\x,\y}\). We may further denote a sequence of vertices by \(\x_1,\x_2,\dots\) without referring to the labels above.

	\paragraph{The model restricted to finite boxes.}
	With the above construction, it is straight-forward to construct the model restricted to finite boxes \((-t^{1/d}/2,t^{1/d}/2]^d\). Instead of constructing the graph on the whole vertex-edge marking, we may simply take the restriction of \(\xi\) to those vertices with locations in \((-t^{1/d}/2,t^{1/d}/2]^d\). Let us denote this restricted vertex-edge marking by \(\xi^t\), that is,
	\[
		\xi^t = \big\{\big((\x,\y),U_{\x,\y}\big)\in\xi\colon x,y\in (-t^{1/d}/2,t^{1/d}/2]^d\big\},
	\]
	and by \(\xi^t_o\) the corresponding Palm version. The subgraph of \(\D\) induced by the box is then simply given by \(\D^{\beta,\gamma,\delta,\Gamma}(\xi^t)\). In the following, we are interested in the model on finite boxes with \emph{periodic boundary conditions}. To this end, let \(\cD^t\) be the functional, acting on vertex-edge markings on \((-t^{1/d}/2,t^{1/d}/2]^d\), that gives a digraph by way of~\eqref{eq:edgeSet} but with the Euclidean metric replaced by the torus metric  
	\[
    	\operatorname{d}_t(x,y):=\min\big\{|x-y+u|:u\in\{-\sqrt[d]{t},0,\sqrt[d]{t}\}^{\times d}\big\}, 
    \]
    where we omit the dependence on the model parameters to avoid confusion. Let us write \(\D^t=\cD^t(\xi^t)\) and, for a given \(\x\in\R^d\times(0,1)\), similarly \(\D^t_{\x}=\cD^t(\xi^t_{\x})\). By local finiteness and amenability, the following lemma is almost immediate.
    
    \begin{lemma}\label{lem:equalNeighbourhoods} 
    	Let \(\x\in\R^d_t\times(0,1)\) be given. Then, almost surely there exists \(T=T(\x)\) such that for all \(t>T\) the in- and out-neighbourhoods of \(\x\) in \(\D^t_{\x}\) and \(\D_{\x}\) coincide. 
    \end{lemma} 
    
    The lemma further implies that the same is true for all \(k\)-neighbourhoods of \(\x\) in the undirected sense. Put differently, for large enough \(t\) the digraphs \(\D^t\) and \(\D\) coincide in any finite graph-neighbourhood of \(\x\). That is, all edges and their orientation coincide up to fixed graph distance.
    
    \paragraph{Graphical construction of directed age-based preferential attachment.}
    We may construct, for each \(t>0\), the graph \(\D_t\) of Section~\ref{sec:DPA} in the same way. Indeed, if we denote by \(\xi_t\) a vertex-edge marking with underlying Poisson process on \((-1/2,1/2]^d\times(0,t]\) (that is all vertices are located in the unit interval and are marked with a birth time in \((0,t]\) and all potentially arcs are marked with an independent uniform random variable), then \(\D_t=\cD^1(\xi_t)\). Let us note here that we may have double used indices of \(\xi\) to indicate both, the addition of vertices and the vertex-edge marking just described. However, as additional vertices are always in bold, there is no danger of confusion. In fact, one may couple all digraphs thus constructed by first taking a larger vertex-edge marking \(\overline{\xi}\) based on a unit-intensity Poisson process on \(\R^d\times(0,\infty)\) and let \(\xi\), \(\xi^t\), and \(\xi_t\) be the respective restrictions of \(\overline{\xi}\). We may therefore still use the common probability measure \(\P\) and expectation operator \(\E\).	
    
    For a fixed \(t>0\), there is a nice connection between the digraphs \(\D^t\), and \(\D_t\). To see this, we define for any \(t>0\) the rescaling map
	\begin{alignat*}{4}
			& h_t: & \big[-\tfrac{1}{2},\tfrac{1}{2}\big)^d\times(0,t) & \longrightarrow \big[-\tfrac{t^{1/d}}{2},\tfrac{t^{1/d}}{2}\big)^d\times (0,1) \\
			& & (y,s) & \longmapsto \big({t^{1/d}y},\tfrac{s}{t}\big).
	\end{alignat*} 
	The rescaling map acts canonically on point sets as well as on edge marks by defining \(h_t(U_{\x,\y})=U_{h_t(\x),h_t(\y)}\) and therefore to vertex-edge markings as well. By the Poisson mapping theorem~\cite{LastPenrose2017}, \(h_t(\xi_t)\) and \(\xi^t\) follow the same law. Moreover, for each \(t_y<t_x<t\), 
	\[
		\rho\Big(\frac{\nicefrac{t_x}{t}\operatorname{d}_t(t^{1/d}x,t^{1/d}y)^d}{\beta\big(\tfrac{\nicefrac{t_x}{t}}{\nicefrac{t_y}{t}}\big)^\gamma}\Big) = \rho\big(\frac{t_x \operatorname{d}_1(x,y)^d}{\beta (\nicefrac{t_x}{t_y})^\gamma}\big) \ \text{ as well as } \ \big(\frac{\nicefrac{t_x}{t}}{\nicefrac{t_y}{t}}\big)^\Gamma = \big(\frac{t_x}{t_y}\big)^\Gamma,
	\]
	 and, consequently, the graphs \(h_t(\D_t) = \cD^t(h_t(\xi_t))\) and \(\D^t\) have the same law. As a result, the graph families \((\D_t:t\geq 0)\) and \((\D^t:t\geq 0)\) have the same one-dimensional marginals. However as a whole, both processes behave differently. Indeed, while the indegree of any fixed vertex \(\x\) stabilises in \(\D^t\) by Lemma~\ref{lem:equalNeighbourhoods} and Theorem~\ref{thm:neighbourhood}, the indegree of a fixed vertex in \(\D_t\) diverges to infinity, cf.~\cite{GGLM2019}. Still, the fact that both sequences have the same one-dimensional marginals is crucial in our proof of Theorem~\ref{thm:DARCMasLL}, i.e.\ DARCM appears as the local limit of the sequence \((\D_t\colon t\geq 0)\). 
	 
	 \subsection{A law of large numbers for independent vertex-edge markings} \label{sec:LLN}
	 In this section, we present a weak law of large numbers for vertex-edge markings that ultimately implies Theorem~\ref{thm:DARCMasLL}. It is based on a law of large numbers for point processes of Penrose and Yukich~\cite[Thm~2.1]{PenroseYukich2003} and its application of Jacob and Mörters to the \emph{spatial preferential attachment model}~\cite[Thm.~7]{JacobMoerters2015}. Although their proof essentially applies to our setting as well, we still give a proof for self-containment. 
	 
	Denote the torus of volume \(t\) by \(\mathbb{T}_t^d=(-t^{1/d}/2,t^{1/d}/2]^d\), endowed with the torus metric \(\d_t\). For convenience, we identify $\R^d$ with $\mathbb{T}_\infty^d$. We consider non-negative functionals $F_t(\x,\xi^t)$ acting on independent vertex-edge markings on $\mathbb{T}_t^d\times(0,1)$, $t\in(0,\infty]$ and a distinguished root vertex \(\x\in\xi^t\). We identify \(F_t(\x,\xi^t)=F_t(\x,\xi^t_{\x})\) whenever \(\x\) is no element of \(\xi^t\). That is, \(\x\) is added to \(\xi^t\) as described above. Such a functional \(F_t\) is called translation invariant, if 
    \[
    	F_t(\x,\xi_{\x}^t)=F_t\circ\theta_x(\x, \xi) \quad \text{for any }x\in \mathbb{T}_t^d,
    \]
    where the point-shift operator $\theta_x: \mathbb{T}_t^d\to \mathbb{T}_t^d$ is given by $\theta_x(z)=z-x, z\in \mathbb{T}_t^d$ and acts on vertex-edge markings by shifting only the locations of the underlying vertex locations but leaving the vertex and edge marks unchanged. Let us additionally write, in a slight abuse of notation, \(\x\in\xi^t\) for a vertex that is part of the underlying marked Poisson process, on which \(\xi^t\) is built.
	
	\begin{prop}[Law of large numbers \cite{JacobMoerters2015}] \label{prop:LLN}
			Let \((F_t:t>0)\) be a a family of non-negative functionals, where \(F_t\) acts translationally invariant on rooted vertex-edge markings on \(\mathbb{T}^d_t\times(0,1)\) and let \(F_\infty\) be a translation invariant functional for a rooted vertex-edge marking on \(\R^d\times(0,1)\). Let us assume that
			\begin{enumerate}[(i)]
				\item \(F_t(\0,\xi^t_o)\to F_\infty(\0,\xi_o)\) in probability, as \(t\to\infty\) and that
				%\item there exists \(p>1\) such that \(\sup_{t>0}\E_o [F_t(\0,\xi^t_0)^p]<\infty\).
				\item the family \((F_t(\o,\xi_o^t)\colon t\geq 0)\) is uniformly integrable. 
			\end{enumerate}
			Then, we have, in probability and in the \(L^1\) sense, that
			\begin{equation*}
				\frac{1}{t}\sum_{\x\in\xi^t}F_t(\theta_x(\x),\theta_x(\xi)) \longrightarrow \E_o[F_\infty(\0,\xi_o)],
			\end{equation*}
			as \(t\to\infty\). 
	\end{prop}
	\begin{proof}
			%We assume without loss of generality that $d=1.$ 
            Let us denote  
			\[
				\overline{F}_t:= \tfrac{1}{t}\sum_{\x\in\xi^t} F_t(\theta_x(\x),\theta_x(\xi)).
			\]  
			By Properties~(i) and~(ii), we have \(F_t(\o,\xi^t_o)\to F_\infty(\o,\xi_o)\) in \(L^1\), and using Campbell's formula~\cite[Prop.~2.7]{LastPenrose2017} together with translation invariance, we further have \(\E \overline{F}_t=\E_o[F_t(\0, \xi^t_0)]\). Hence, 
			\[
				\lim_{t\to\infty} \E \overline{F}_t = \E_o[F_\infty(\0,\xi_0)]<\infty.
			\]
			Let us next assume that the \(F_t\) form a uniformly bounded sequence, which particularly implies that the \(F_t\) are uniformly integrable in \(L^p\) for any \(p>2\). We have for the second moment 
			\[
				\E\overline{F}_t^2 = \frac{1}{t^2}\E \Big[\sum_{\x\in\xi^t} F_t(\theta_x(\x),\theta_x(\xi))^2\Big] +\frac{1}{t^2} \E \Big[\sum_{\x\neq \y \in\xi^t} F_t(\theta_x(\x),\theta_x(\xi))F_t(\theta_y(\y),\theta_y(\xi))\Big].
			\]	
			As all second moments exist (under our boundedness assumption), the first term on the right-hand side converges to zero by the same calculation performed for the first moment, while the second term reads
			\[
				\begin{aligned}
					\int_{\bbT^d_t} \frac{\d x}{t} 
					& 
						\int_0^1 \d s \int_{\bbT_t^d}\frac{\d y}{t} \int_0^1 \d t \, \E_{\x,\y}[F_t((o,s),\theta_x(\xi_{\x,\y}))F_t((o,t),\theta_y(\xi_{\x,\y}))]
					\\ &
						= \E\big[F_t((o,S),\theta_X(\xi_{\X,\Y}))F_t((o,T),\theta_Y(\xi_{\X,\Y}))\big],
				\end{aligned} 
			\]	
			where \(\X=(X,S)\) and \(\Y=(Y,T)\) are independently chosen uniformly from \(\bbT^d_t\times(0,1)\). Observe that \(\theta_X(\xi_{\X,\Y}))\) is the same as \(\theta_X(\xi_{\X})) \cup \{(Y-X,T)\}\), where the latter denotes the vertex-edge marking where the vertex \(\{(Y-X,T)\}\) has been added. Consider the event \(\cE_t=\{\d_t(X,Y)>\sqrt[2d]{t}\}\), which holds with probability arbitrarily close to one for sufficiently large \(t\). By boundedness of \(\sup_t F_t\), we hence have 
			\[
				\E\big[F_t((o,S),\theta_X(\xi_{\X,\Y}))F_t((o,T),\theta_Y(\xi_{\X,\Y})) \1_{\cE_t^\mathrm{c}}\big] \longrightarrow 0,
			\]
			as \(t\to\infty\) and we work conditionally on the event \(\cE_t\) in the following. We introduce further the event \(\cG_t=\{\exists \z\in\xi^t:\d_t(o,z)>\sqrt[2d]{t}\}\), which holds with probability arbitrarily close to one for large enough \(t\), as well. Let us write \(F'_{\sqrt[2d]{t}}(S,\theta_X(\xi_{\Y}))\) for the restriction of the functional \(F_t(S,\theta_X(\xi_{\Y}))\) to the vertices in \(\bbT_{\sqrt[2d]{t}}^d\times(0,1)\). We make the following important observations:
			\begin{itemize}
				\item Conditioned on \(\cE_t\), the restriction of \(\theta_X(\xi_{\X,\Y})\) to \(\bbT_{\sqrt[2d]{t}}^d\times(0,1)\) coincides with the same restriction of \(\theta_X(\xi_{\X})\).
				\item The law of \(\theta_X(\xi_{\X,\Y})\), given \(\cE_t\), equals the law of \(\theta_X(\xi_{\X})\), given \(\cG_t\).
			\end{itemize}
			Combined we observe
			\begin{equation*}
				\begin{aligned}
					\E\big[
					&
						|F_t((o,S), \theta_X(\xi_{\X,\Y}))-F'_{\sqrt[2d]{t}}((o,S),\theta_X(\xi_{\X}))| \,\big| \, \cE_t\big]
					\\ &
						= \E\big[|F_t((o,S),\theta_X(\xi_{\X,\Y}))-F'_{\sqrt[2d]{t}}((o,S),\theta_X(\xi_{\X,\Y}))| \,\big| \, \cE_t\big]	
					\\ &
						= \E\big[|F_t((o,S),\theta_X(\xi_{\X}))-F'_{\sqrt[2d]{t}}((o,S),\theta_X(\xi_{\X}))| \,\big| \, \cG_t\big]		
					\\ &
						\longrightarrow 0,	
				\end{aligned}
			\end{equation*}
			as \(t\to\infty\), by Property~(i) and bounded convergence. In the same way, we obtain
			\[
				\E\big[|F_t((o,T),
						\theta_Y(\xi_{\X,\Y}))-F'_{\sqrt[2d]{t}}((o,T),\theta_Y(\xi_{\Y}))| \,\big| \, \cE_t\big] \longrightarrow 0.
			\]
			Hence, using the boundedness of \(\sup_t F_t\) once more, we obtain
			\[
				\begin{aligned}
					\E\big[
					&
						F_t((o,S), \theta_X(\xi_{\X,\Y}))F_t((o,T),\theta_Y(\xi_{\X,\Y})) \, \big| \,  \cE_t \big] 
					\\ &	
						= \E\big[F'_{\sqrt[2d]{t}}((o,S),\theta_Y(\xi_{\X}))F'_{\sqrt[2d]{t}}((o,T),\theta_Y(\xi_{\Y})) \, \big| \cE_t\big] + o(1) 
					\\ &
						= \E\big[F'_{\sqrt[2d]{t}}((o,S),\theta_Y(\xi_{\X})) \, \big| \, \cE_t\big] \E\big[F'_{\sqrt[2d]{t}}((o,T),\theta_Y(\xi_{\Y})) \, \big| \cE_t\big] +o(1), 
					\\ &
						\longrightarrow \E_o[F_\infty(\0,\xi_0)]^2
				\end{aligned}
			\]
			as the two \(F'_{\sqrt[2d]{t}}\) functionals are conditionally independent, given \(\cE_t\), and both converge in probability to \(F_\infty(\0,\xi)\) by Property~(i). Hence, the convergence in the last line is a consequence of bounded convergence. This then proves convergence of \(\overline{F}_t\) to \(\E_o[F_\infty(\0,\xi_0)]\) in \(L^2\) for uniformly bounded functionals \(F_t\) that satisfy Property~(i). 
			
			Finally, for arbitrary functionals \(F_t\) satisfying all assumptions, we introduce the uniformly bounded functionals \(F_t^k:= F_t\wedge k\). Clearly, \(F_t^k\) fulfils Property~(i) and therefore \(\overline{F^k_t}\to \E_o F_\infty^k(\0,\xi_o)\) in \(L^2\). Further,
			\[ 				
				\begin{aligned}
					0 
					& 
						\leq \E\Big[\frac{1}{t}\sum_{\x\in\cX^t}\big(F_t(\theta_x(\x),\theta_x(\xi))-F_t^k(\theta_x(\x),\theta_x(\xi))\big)\Big] 
					%\\&
						= \E_o\big[F_t(\0,\xi_o^t)-F_t^k(\0,\xi_o^t)\big] \longrightarrow 0,
				\end{aligned}
			\]  
			uniformly as \(k\to\infty\) by uniform integrability. Hence, \(\overline{F}_t\) converges in \(L^1\) to
			\[
				\lim_{k\to\infty} \E_o F^k_\infty(\0,\xi_o) = \E_o F_\infty(\0,\xi_o),
			\]
			which concludes the proof. 
		\end{proof}
		
		The LLN immediately implies that DARCM is the local limit of the directed age-based preferential attachment sequence.
		 
		\begin{proof}[Proof of Theorem~\ref{thm:DARCMasLL}]
			Note first, that \(\sharp\scrV(\D_t)\sim t\), almost surely, by the law of large numbers for Poisson processes. We therefore may replace the denominator \(\sharp\scrV(\D_t)\) in the statement of Theorem~\ref{thm:DARCMasLL} by \(t\). Let \(H\) be a non-negative functional acting on a bounded graph neighbourhood of the root of a rooted digraph that fulfils the uniform integrability condition of Theorem~\ref{thm:DARCMasLL}. Since \(\D_t\) and \(\D^t=\cD^t(\xi^t)\) have the same law, it suffices to show~\eqref{eq:DARCMasLL} for \(\D^t\). To this end, choose \(F_t(\x,\xi^t_{\x})=H(\x,\cD^t(\xi^t_{\x}))\). Property~(i) of Proposition~\ref{prop:LLN} is a direct consequence of Lemma~\ref{lem:equalNeighbourhoods} and the fact that \(H\) only depends on a finite graph neighbourhood of the root, while Property~(ii) is satisfied due to the uniform integrability assumption on \(H\). The proof concludes by applying Proposition~\ref{prop:LLN}.
		\end{proof}

		\begin{remark}
			While Theorem~\ref{thm:DARCMasLL} is formulated in the classical notion of uniformly integrable , its proof uses the LLN for vertex-edge markings, which allows for stronger statements. Specifically, one may apply functionals that depend on the whole graph and still apply the limit as long as the convergence assumption in (i) is satisfied. Moreover, Proposition~\ref{prop:LLN} allows that \(F_t\) depends on the lengths of edges or vertex marks in the graph \(\D^t\), so one may infer results for rescaled edge lengths or rescaled birth times in \(\D_t\). The reason why we achieve stronger limit results than typically possible in random graph theory lies in the fact that the underlying geometry provides more structure than the usual local topology of sparse random graphs. In the following, we mainly work with the law of large numbers rather than the local limit theorem. The reason is simply that formulating the assumptions via a Palm measure keeps notation more concise. However, we do not apply it to any situation where the stronger formulation is required and all proofs can be performed, \emph{mutatis mutandis}, using Theorem~\ref{thm:DARCMasLL}. Situations where the stronger law of large numbers actually is of importance are discussed in Section~\ref{sec:outlook}.		
		\end{remark}

\subsection{Degree distribution and sparsity} \label{sec:proofDeg}
In this section, we give the proofs for the degree distribution stated in Section~\ref{sec:degrees} and show sparsity of the digraph family \((\D_t)_t\).

\begin{proof}[Proof of Theorem~\ref{thm:neighbourhood}.]
The incoming arcs of \(\0\) are all edges to younger neighbours of \(\o\) in \(\G=\D^{\beta,\gamma,\delta,1}(\xi_o)\) plus the edges to older neighbours where a conversely oriented edge has been added. From~\cite[Lemma 4.4]{GGLM2019}, the number of younger neighbours in \(\G\) from \((o,u)\) (i.e.\ the root's mark is given \(U_o=u\)) is heavy-tailed with power-law exponent \(1+1/\gamma\). The older incoming neighbours of \((o,u)\) form a Poisson process on \(\R^d\times (0,u)\) with intensity
\[
    \pi(u/s)\rho(\beta^{-1}s^{\gamma}u^{1-\gamma}|x|^d) \ \mathrm{d}x \ \mathrm{d} s.
\]
Since \(\pi(u/s)\leq 1\), the number of such neighbours is at most Poisson distributed with parameter \(\nicefrac{\beta\omega_d\delta}{(1-\gamma)(\delta-1)}\)~\cite[Proposition 4.1(c)]{GGLM2019}. Hence, the indegree of \(\0\) in \(\mathscr{D}\) is bounded from below by the number of younger neighbours of \(\0\) in \(\G\) and from above by the number of younger neighbours of \(\0\) in \(\G\) plus an independent Poisson distributed random variable. Therefore, the heavy-tailed distributed number of young neighbours in \(\G\) dominates, proving~(a).

Similarly, the outgoing neighbours of \(\0\) in \(\mathscr{D}\) are the older neighbours of \(\0\) in \(\G\) plus the younger ones where an additional conversely oriented edge has been added. The number of the first type is again Poisson distributed with parameter \(\nicefrac{\beta\omega_d\delta}{(1-\gamma)(\delta-1)}\), independent of the root's mark. For fixed mark \(U_0=u\), the latter form a Poisson process on \(\R^d\times(u,1)\) with intensity
\[
    \pi(s/u)\rho\big(\beta^{-1}u^\gamma s^{1-\gamma}|x|^d\big) \ \mathrm{d}x \ \mathrm{d}s.
\]
The expected number of such vertices is hence
\begin{equation}\label{eq:degCalc1}
    \begin{aligned}
        \int_u^1 \mathrm{d}s \, \pi(u/s) \int_{\R^d}\mathrm{d}x \ \rho(\beta^{-1}u^{\gamma}s^{1-\gamma}|x|^d) 
        & 
        	=\beta u^{\Gamma-\gamma} \int_u^1 s^{\gamma-\Gamma-1} \d s \int_{\R^d} (1\wedge |x|^{-d\delta}) \d x
        \\ &
        	\asymp \frac{\omega_d \delta}{\delta-1}\cdot \frac{\beta}{|\gamma-\Gamma|} u^{\Gamma-\gamma} (1\vee u^{\gamma-\Gamma}) 
        \\&
        	=\frac{\omega_d \delta}{\delta-1}\cdot \frac{\beta}{|\gamma-\Gamma|} (1\vee u^{-\gamma+\Gamma}),
    \end{aligned}
\end{equation}
where we assume for now that \(\gamma\neq \Gamma\). Hence, if \(\Gamma>\gamma\), this gives a Poisson distribution with parameter \(\Theta(\frac{\omega_d \delta}{\delta-1}\,\frac{\beta}{|\gamma-\Gamma|})\). Since the Poisson points with birth time smaller than \(u\) are independent of those with birth time greater than \(u\), the claim follows by the convolution property of the Poisson distribution. If \(\Gamma<\gamma\), then the outdegree is mixed Poisson distributed and, writing \(c=\tfrac{\omega_d \delta}{\delta-1}\cdot\tfrac{\beta}{\gamma-\Gamma}\), we have
\begin{equation*}
    \begin{aligned}
        \P_{o}\{\sharp\scrN^{\text{out}}=k\} 
        & 
        	\asymp \int\limits_0^1 \exp\big(c u^{\Gamma-\gamma}\big) \frac{(c u^{\Gamma-\gamma})^k}{k!} \mathrm{d} u 
        	\asymp \frac{1}{\Gamma(k+1)}\int_0^\infty e^{-\lambda} \lambda^{k-\nicefrac{1}{(\gamma-\Gamma)}-1} \mathrm{d}\lambda 
        \\ &	
            \asymp k^{-1-1/(\gamma-\Gamma)},
    \end{aligned}
\end{equation*}
as \(k\uparrow\infty\), using Stirling's formula in the last step. Finally, if \(\gamma=\Gamma\), then the integral \eqref{eq:degCalc1} instead calculates to 
\begin{equation}\label{eq:degCalc2}
	\frac{\omega_d \delta}{\delta-1}\beta \log(1/u).
\end{equation} 
Therefore, the outdegree is again mixed Poisson distributed and
\[
	\begin{aligned}
		\P_{o}\{\sharp\scrN^{\text{out}}=k\}
		&
			\asymp 	\int_0^1 \frac{u\log(1/u)^k}{k!}\d u.
	\end{aligned}
\]
Integration by parts yields
\[
	\int_0^1 \frac{u\log(1/u)^k}{k!}\d u = \frac{1}{2} \int_0^1 \frac{u\log(1/u)^{k-1}}{(k-1)!} \d u
\]
The proof concludes by repeating this step \(k-1\) times. 
\end{proof}

We conclude this section by proving sparsity for \((\D_t)_t\).

\begin{proof}[Proof of Theorem~\ref{thm:sparse}.]
	Note that it suffices to show the result for the family \((\D^t)_t\). Let \(F(\x,\xi^t)\) be the outdegree of vertex \(\x\) in \(\cD^t(\xi_{\x})=\D^t_{\x}\). Then, by Theorem~\ref{thm:neighbourhood} and Lemma~\ref{lem:equalNeighbourhoods}, we have
	\[
		\sup_{t}\E_o[F(\o,\xi_{\o})^p]<\infty 
	\]  
	for some small enough \(p>1\), implying uniform integrability. Therefore the LLN, Proposition~\ref{prop:LLN}, applies and we infer
	\[
		\frac{1}{\sharp\scrV(\D^t)}\sum_{\x\in\D^t} F(\theta_x\x, \theta_x\xi^t) \longrightarrow \E_o[F(\o,\xi_o)] = \E_o\big[\sharp\scrN^{\text{out}}\big],
	\]
	in probability, as \(t\to\infty\), where we used that \(\sharp\scrV(\D^t)\sim t\), almost surely. 
\end{proof}

%%%%%%%%%%%%%%%%%%%%%%%%%%%
%%%%%% Clustering %%%%%%%%%
%%%%%%%%%%%%%%%%%%%%%%%%%%% 
\subsection{Proof of clustering results} \label{sec:proofs_clust}
In this section, we mainly focus on the results about interest clustering. The average friend clustering results can be obtained in the same way as the average triangle count in~\cite[Thm.~5.1]{GGLM2019} with the probability of existence of an edge being replaced by the existence of a double-arc. In the same vein, the proof for the global friend clustering coefficient applies. However, as it is more dependent on the new degree distributions, we give the proof for completeness.
\begin{proof}[Proof of Theorem~\ref{thm:globFriendClust}.]
	Let us count the number of open and closed triangles separately (triangles are always considered to be bidirectional during the proof) and we may work in the rescaled graph \(\D^t=\cD^t(\xi^t)\). Let \(F(\x,\xi^t)\) be the number of triangles in \(\D^t_{\x}\) that have their youngest vertex in \(\x\). This number is bounded by the square of the number of all outgoing arcs that have been formed when \(\x\) was born. The number of such arcs (without being squared) is Poisson distributed by~\cite[Prop.~4.1]{GGLM2019} and the proof of Theorem~\ref{thm:neighbourhood}, respectively. Therefore, the law of large numbers, Proposition~\ref{prop:LLN} applies and yields \(t^{-1} \sum F(\theta_x(\x),\theta_x(\xi^t)))\to \E_o[F(o,\D_o)]\), in probability. 
	
	Similarly, let \(G(\x,\xi^t)\) be the number of open triangles in \(\D^t\), in which both double-arcs are incident to \(\x\). Let \(\cN_>^{\text{out}}(\x)\) be the number of out-neighbours older than \(\x\) (which are all out-arcs formed at the birth of \(\x\)), and let \(\cN_<^{\text{out}}(\x)\) the number of out-neighbours older than \(\x\) (which are all reciprocal arcs formed). There are now three possibilities: either \(\x\) is the youngest vertex, the middle vertex, or the oldest vertex in the open triangle. Therefore, 
	\[
		\begin{aligned}
			\tfrac{1}{2}\cN_<^{\text{out}}(\x)(\cN_<^{\text{out}}-1) \leq G(\x,\xi^t)
			&
				\leq \cN_>^{\text{out}}(\x)^2 + \cN_>^{\text{out}}(\x)\cN_<^{\text{out}}(\x) + \cN_<^{\text{out}}(\x)^2.
		\end{aligned}
	\]  
	By Theorem~\ref{thm:neighbourhood} and its proof, \(\cN_>^{\text{out}}(\x)^{2}\) as well as \(\cN_>^{\text{out}}(\x)\cN_<^{\text{out}}(\x)\) are uniformly integrable, while \(\cN_<^{\text{out}}(\x)^{2}\) is uniformly integrable if and only if \(1/2>\gamma-\Gamma\). In that case, we may again apply the law of large numbers and infer \(t^{-1} \sum G(\theta_x(\x),\theta_x(\xi^t)))\to \E_o[G(o,\D_o)]\), in probability. Combined, we deduce that the global clustering coefficient converges in probability towards a positive constant. In the other case, \(1/2\leq \gamma-\Gamma\), we apply the law of large numbers to the bounded functionals \(G\wedge k\) and infer \(t^{-1} \sum G(\theta_x(\x),\theta_x(\xi^t)))\to \infty\) by sending \(k\to\infty\). This yields \(c^{\textup{fc}}_{\textup{glob}}(\D_t)\to 0\), in probability as \(t\to\infty\). 
\end{proof} 

Let us turn to the more involved proofs for the interest clustering coefficients. We start with the proof for the global version, i.e.\ Theorem~\ref{thm:globIC}.

\begin{proof}[Proof of Theorem~\ref{thm:globIC}]
As the key ingredient of our proof, we derive orders for the expected number of open and closed bow-ties containing the origin \(\o=(o,t_o)\) (recall Figure~\ref{fig:interest}) at time \(t\) in the rescaled graph \(\D^t_{\o}\) on \(\bbT^d_t\times(0,1)\). Let us write \(\mu_t^o(t_o)\) and \(\mu_t^c(t_o)\) for the expected number of open and closed bow-ties containing \(\o\).  Let us additionally define the truncated expected number of open bow-ties
\begin{equation}\label{eq:defS}
    S(t_o) \;:=\; \E_{\o}\sum_{\substack{\mathbf{u},\,\mathbf{v},\,\mathbf{w}\,\in\,\D_t\\ t_u,t_v,t_w>1/(t\log t)}}
    \1\big\{\0\to\mathbf{v}\text{ in }\D_t\big\}\;
    \1\big\{\0\to\mathbf{u}\text{ in }\D_t\big\}\;
    \1\big\{\mathbf{w}\to\mathbf{v}\text{ in }\D_t\big\},
\end{equation}
where we discard vertices with marks lower than $1/(t\log t)$. Using the three-fold Mecke equation and the facts that \(\bbT_t^d\) has volume \(t\) and that all indicators are bounded by \(1\), we obtain 
\begin{equation*}
    \mu_t^o(t_o) \asymp S(t_o) =\int\limits_{(\mathbb{T}_t^d\times(\nicefrac{1}{t\log t},1))^3} \!\!\! \P_{\o,\u}(\0\to\mathbf{u})\;\P_{\o,\v}(\0\to\mathbf{v})\;\P_{\w,\v}(\mathbf{w}\to\mathbf{v})\d\mathbf{u}\d\mathbf{v}\d\mathbf{w},
\end{equation*}
where $d\mathbf{u}=\d u \d t_u$ etc. Similarly, we define 
\begin{equation} \label{eq:defSclosed}
S^c(t_o)
:=
\E_{\0}\!\!\sum_{\substack{\u,\v,\w\in\scrX_t\\ t_u,t_v,t_w\ge 1/(t\log t)}}
\1\{\0\to\u\}\,\1\{\0\to\v\}\,\1\{\w\to\u\}\,\1\{\w\to\v\},
\end{equation}
and obtain 
\[
    \mu_t^c \asymp \int\limits_{(\bbT_t^d\times(\nicefrac{1}{t\log t},1))^3}\!\!\!
\P_{\o,\u}(\0\to\u)\,\P_{\o,\v}(\0\to\v)\,\P_{\w,\u}(\w\to\u)\,\P_{\w,\v}(\w\to\v)\d \u \d \v\d \w.
\]
In particular, the asymptotic behaviour of the expected number of both, open and closed bow-ties are driven by those vertices born after \(1/(t\log t)\) (in fact with high probability no vertex was born before) and it suffices to consider \(S(t_o)\) and \(S^c(t_o)\), respectively, which is done in the Lemmas~\ref{lem:orderSopen} and~\ref{lem:orderSclosed}, respectively. 

Let us first assume \(\gamma<1/2\). Then, by Lemma~\ref{lem:orderSopen}, we have that the expected number of open bow-ties \(\mu_t^o:= \int_0^1\mu_t^o(t_o)\d t_o\), incident to \(\o\) in the rescaled graph \(\D_o^t\), is uniformly bounded (in \(t\)) and the same is true for \(\mu_t^c:= \int_0^1\mu_t^c(t_o)\d t_o\). Let us consider the functionals
\[
F^o(\x, \D^t)=\sum_{\u,\v\in\D^t}\1\{\x\to \u, \x\to \v\}\sum_{y\in\D^t}\1\{\y\to\u \text{ or } \y\to\v\},
\] 
and
\[ 
    F^c(\x,\D^t)=\sum_{\u,\v\in\D^t}\1\{\x\to \u, \x\to \v\}\sum_{y\in\D^t}\1\{\y\to\u, \y\to\v\}.
\]
We aim to showing the positivity of the limiting global interest coefficient by using the simple expansion
\begin{equation*}
\frac{\sum_{\x}F^c_t(\x,\D^t) }{\sum_{\x}F^o_t(\x,\D^t)}=\frac{\sum_{\x}F^c_t(\x,\D^t) }{\sharp\mathscr{V}(\D^t)}\;\frac{\sharp\mathscr{V}(\D^t)}{\sum_{\x}F^o_t(\x,\D^t)}.
\end{equation*}
Consider \(F^*(\o,\D^t_o)\) for \(*\in\{o,c\}\). As this quantity only depends on a bounded graph-neighbourhood of the root, we have \(F^*(\o,\D^t_o)\to F^*(\o,\D_o)\), almost surely by Lemma~\ref{lem:equalNeighbourhoods}.  
The corresponding $L^1$-convergence, or equivalently uniform integrability, can then simply be checked by convergence of the first moments, which in turn can be checked by uniform boundedness of the first moments as well as the existence of the limiting expectation. On the first condition, we have already commented above while the second condition follows from the same calculation. Hence,
\[
    \frac{\sum_{\x}F^c(\x,\D^t) }{\sum_{\x}F^o(\x,\D^t)} \longrightarrow \frac{\E_o F^c(\o,\D_o)}{\E_o F^o(\o,\D_o)}>0,
\]
in probability and \(L^1\), as \(t\to\infty\), by Proposition~\ref{prop:LLN}. 

Consider now the case \(\gamma\geq 1/2\). In that case, we always have \(\mu_t^o\to\infty\). However, the Lemmas~\ref{lem:orderSopen} and~\ref{lem:orderSclosed} imply that \(\mu_t^c/\mu_t^o \to 0\) even in the cases where \(\mu_t^c\to\infty\), as well. Recall that we work in the rescaled graph \(\D^t\) and observe that a simple Poisson calculation yields that the event
	\[
		\cE_t:=\{\text{the oldest vertex in }\D^t\text{ was born after }1/(t\log(t))\}
	\]
	holds with high probability, i.e.\ with probability converging to one as \(t\to\infty\). We therefore work on \(\cE_t\) in the following. Let us define
	\begin{equation} \label{eq:defY_t}
		Y_t :=\sum_{\substack{\x,\u,\v,\w\in\mathcal D_t\\ t_{\x},t_{\u},t_{\v},t_{\w}>1/(t\log t)}} \1\{\x\to\u\}\,\1\{\x\to\v\}\,\1\{\w\to\v\},
	\end{equation} 
	and 
	\[
		X_t :=\sum_{\substack{\x,\u,\v,\w\in\mathcal D_t\\ t_{\x},t_{\u},t_{\v},t_{\w}>1/(t\log t)}} \1\{\x\to\u\}\,\1\{\x\to\v\}\,\1\{\w\to\u\}\,\1\{\w\to\v\}.
	\]
Put differently, $Y_t$ counts \emph{open} bow-ties and $X_t$ counts \emph{closed} bow-ties under the birth-time truncation. Clearly, $0\le X_t\le Y_t$ pointwise. Moreover, the denominator in the definition of the global interest clustering coefficient dominates $Y_t$ (since $\1\{\w\to\u\ \text{or}\ \w\to\v\}\ge \1\{\w\to\v\}$), and hence (on \(\cE_t\))
\[
	0\le c^\mathrm{ic}_\mathrm{glob}(\mathcal \D^t)\le \frac{X_t}{Y_t}, \qquad \text{with the convention } \frac{X_t}{Y_t}:=0\ \text{on }\{Y_t=0\}.
\]
Using Mecke's equation,
\[
	\E Y_t = t\int_{1/(t\log t)}^1 S(t_o)\,\d t_0, \qquad \E X_t = t\int_{1/(t\log t)}^1 S^c(t_o)\,\d t_0,
\]
and therefore our previous observations imply \((\E X_t)/(\E Y_t)\to 0\).
Fix $\varepsilon\in(0,1)$. Splitting according to the event $\{Y_t\ge \varepsilon\,\E Y_t\}$ that occurs \emph{with high probability} by Lemma~\ref{lem:open-bowtie-conc} and using $X_t/Y_t\leq 1$ yields
\begin{equation*} %\label{eq:ratio-reduction}
	\E\Big[\frac{X_t}{Y_t}\Big] \le \frac{\E X_t}{\varepsilon\,\E Y_t} + \P\big(Y_t<\varepsilon\,\E Y_t\big)=o(1),
\end{equation*}
by Lemma~\ref{lem:open-bowtie-conc} and the previous observations. Consequently, for any \(\varepsilon>0\),
\[
	\begin{aligned}
		\P\big(c^{\textup{ic}}_{\textup{glob}}(\D^t)>\varepsilon\big) \leq \P(\cE_t^\mathrm{c}) + \P(X_t/Y_t>\varepsilon) \longrightarrow 0,
	\end{aligned}
\]
as \(t\to\infty\). This concludes the proof. 
\end{proof}

We close this clustering-proof section with the results about the average interest clustering coefficient.

\begin{proof}[Proof of Theorem~\ref{thm:locIC}.]
We treat the cases \(\gamma<1/2\) and \(\gamma\geq 1/2\) separately. 
\paragraph{Case {$\boldsymbol{\gamma<1/2}$}.}
We make use of our local limit structure. Consider a representation of \(c^\text{ic}_\text{av}(\D_t)\) from the viewpoint of a typical vertex. Write
\[
	\sharp\scrI_t = \sum_{\x\in \D_t}\sum_{\y\in\D_t} \1_{\{(\x,\y)\in\scrI_t\}} =: \sum_{\x\in\D_t} f_{\scrI}(\x, \D_t), \quad \text {where }f_{\scrI}(\x,\D_t):=\sum_{\y\in\D_t}\1_{\{\{\x,\y\}\in\scrI_t\}}.
\]	
We may further write
\[
	\begin{aligned}
		\sum_{(\x,\y)\in\scrI_t} c^\text{ic}((\x,\y),\D_t) 
		&
			= \sum_{\x\in\D_t} f_c (\x, \D_t), \quad \text{ where } f_c(\x,\D_t):= \sum_{\y\in\D_t } c^{\text{ic}}((\x,\y),\D_t),
	\end{aligned}
\]
since \(c^{\text{ic}}((\x,\y),\D_t)=0\), if \((\x,\y)\not\in\scrI_t\). We have 
\[
f_c(\x,\D_t)\leq f_{\scrI}(\x, \D_t)\leq \mathbf{D}^{(2)}(\x,\D_t),
\]
where the latter random variable counts all vertices at distance two of $\x$ if edge orientation is disregarded. If $\gamma<1/2$, then the $\mathbf{D}^{(2)}(\x,\D_t)$ are uniformly integrable by \cite[Prop.~4.1 \& Lem.~4.4]{GGLM2019}. Consequently, we have by Theorem~\ref{thm:DARCMasLL},
\[
	\frac{\sharp\D_t}{\sharp\scrI_t} \longrightarrow \frac{1}{\E_o\big[\1{\{\sharp\scrN^\text{out}\geq 2\}} \sharp\big\{\y\colon \scrN^\text{out}\cap\scrN^\text{out}(\y)\neq \emptyset\big\}\big]}
\]
and 
\[
	\frac{1}{\sharp \D_t} \sum_{\x\in\D_t} f_c(\x,\D_t) \longrightarrow \E_o\Big[\1{\{\sharp\scrN^\text{out}\geq 2\}}\sum_{\y\in\D}c^\text{ic}((\o,\y),\D_o)\1{\big\{\scrN^\text{out}\cap \scrN^\text{out}(\y)\neq \emptyset\big\}}\Big],
\]
in probability as \(t\to\infty\), where we used that \(c^\text{ic}((\o,\y),\D_o)>0\) implies \(\scrN^\text{out}\cap \scrN^\text{out}(\y)\neq \emptyset\) as well as \(\sharp\scrN^\text{out}\geq 2\). Combining both result finishes the proof for $\gamma<1/2$.

\paragraph{Case \(\boldsymbol{\gamma\geq 1/2}\).} 
Throughout this part of the proof, the sums over pairs of vertices are understood to run over \emph{ordered pairs of distinct vertices}. We may still work in the rescaled graph \(\D^t\) under the birth-time truncation at \(1/(t\log t)\). Write
\begin{equation}\label{eq:defMN}
	\bfM_t:=\sum_{\x\in \D^t} f_\scrI(\x,\D^t),
		\qquad
	\bfN_t:=\sum_{\x\in \D^t} f_c(\x,\D^t),
\end{equation}
so that $\mathrm{c}^{\mathrm{ic}}_{\mathrm{av}}(\D^t)=\bfN_t/\bfM_t$ on the event $\{\bfM_t>0\}$. Note that $\E[\bfM_t]\to\infty$ since \(\gamma\geq 1/2\) by Lemma~\ref{lem:orderMN}. Thus, $\P(\bfM_t=0)\to0$ and we may work exclusively on the complement. We now argue as in the proof of Theorem~\ref{thm:globIC}. First, by Lemma~\ref{lem:orderMN}, we have \(\E \bfN_t/\E\bfM_t=o(1)\).
Secondly, a second-moment argument, analogous to Lemma~\ref{lem:open-bowtie-conc} (expanding $\bfM_t^2$ and applying the second-order Mecke formula, with overlap configurations bounded using the same open/closed bow-tie estimates), yields $\Var(\bfM_t)=o\big(\E[\bfM_t]^2\big)$ and hence, in probability, \({\bfM_t}/{\E[\bfM_t]}\to1\). Therefore, for every $\varepsilon>0$,
\[
	\P({\bfN_t}/{\bfM_t}>\varepsilon) \leq  \P(\bfM_t\leq \E[\bfM_t]/2) + \frac{2\,\E[\bfN_t]}{\varepsilon\,\E[\bfM_t]}
%\leq \P\!\left(M_t\le \frac12\E[M_t]\right)
%+
%\frac{4\,\E[Z_t]}{\varepsilon\,\E[M_t]}
\longrightarrow 0,
\]
as \(t\to\infty\), thus $\mathrm{c}^{\mathrm{ic}}_{\mathrm{av}}(\D^t)\to 0$ in probability, for $\gamma\geq 1/2$. This concludes the proof.
\end{proof}

%%%%%%%%%%%%%%%%%%%
%%%%%% Perc %%%%%%%
%%%%%%%%%%%%%%%%%%%

\subsection{Percolation proofs}\label{sec:proof_perc}

In this section, we give the proofs of Section~\ref{sec:perc}. The main ideas are largely inspired by~\cite{GLM2021}. There, it is shown that a promising strategy to build an (undirected) infinite path in \(\G\) (i.e.\ \(\Gamma=0\)) for all positive intensities is to connect powerful or old vertices via young connectors whenever \(\gamma>\nicefrac{\delta}{(\delta+1)}\). This strategy gets harder in our directed setting because it is harder to reach a young connector from an old vertex with an arc than it has been in the undirected setting. Our first result therefore concerns the probability, that two relatively old vertices are connected via a young connector. 

In accordance with~\cite{GLM2021}, we write \(\x \overset{2}{\underset{\x,\y}{\rightsquigarrow}}\y\) for the event that \(\x\) is connected to \(\y\) by a directed path of length \(2\) where the intermediate vertex is younger than both \(\x\) and \(\y\). 

\begin{lemma}[Two-connection-lemma] \label{lem:twoConnect}
Let \(\D\) be the directed age-dependent random connection model with parameter \(\beta>0,\delta>1, \gamma\in(0,1)\) and \(\Gamma\geq 0\).  
\begin{enumerate}[(a)]
    \item Assume \(\gamma<\nicefrac{(\delta+\Gamma)}{(\delta+1)}\). Let \(\x=(x,t)\) and \(\y=(y,s)\) be two given vertices with \(s<t\) that further satisfy \(|x-y|^d\geq \beta s^{-\gamma}t^{\gamma-1}\). Then we have
        \begin{equation}\label{eq:twoConnectSub}
           \P_{\x,\y}\Big\{\x \overset{2}{\underset{\x,\y}{\rightsquigarrow}}\y\Big\}\leq \E_{\x,\y}\big[\sharp\{\z=(z,u):u> t\text{ and }\x\rightsquigarrow\z\rightsquigarrow\y\}\big] \leq \beta C \; \P_{\x,\y}\{\x\rightsquigarrow\y\}
        \end{equation}
    as well as
        \begin{equation}\label{eq:twoConnectSub2}
           \P_{\x,\y}\Big\{\y\overset{2}{\underset{\x,\y}{\rightsquigarrow}}\x\Big\}\leq \E_{\x,\y}\big[\sharp\{\z=(z,u):u> t\text{ and }\y\rightsquigarrow\z\rightsquigarrow\x\}\big] \leq \beta C \; \P_{\x,\y}\{\y\rightsquigarrow\x\},
        \end{equation}
    where in both cases \(C=\tfrac{2^{d\delta+1}\omega_d \delta}{d(\delta-1)(\delta(1-\gamma)+\gamma-\Gamma)}\).
    \item If \(\gamma>\nicefrac{(\delta+\Gamma)}{(\delta+1)}\), we choose the two constants
        \[
            \alpha_1\in\big(1,\tfrac{\gamma-\Gamma}{\delta(1-\gamma)}\big) \text{ and then }\alpha_2\in\big(\alpha_1, \tfrac{\alpha_1(\gamma\delta-1)+\gamma-\Gamma}{\delta-1}\big).
        \]
    Let \(\x=(x,t)\) be a given vertex with \(t<\nicefrac{1}{2}\) and define the event 
    \[
        \cE(\x)=\Big\{\exists\y=(y,s):s<t^{\alpha_1}, |x-y|^d < t^{-\alpha_2} \text{ and } \x \overset{2}{\underset{\x,\y}{\rightsquigarrow}}\y \Big\}.
    \]
    For all \(\beta>0\) there exists some \(a>0\) such that
    \begin{equation*}%\label{eq:twoConnetSup}
        \P_{\x}\big(\cE(\x)\big) \geq 1-\exp(-t^{-a}).
    \end{equation*}
\end{enumerate}
\end{lemma}
\begin{proof}
    We start by proving (a). Observe that the first inequality in~\eqref{eq:twoConnectSub} is simply a moment bound, and we hence focus on the second inequality. By our assumptions on the distance of \(\x\) and \(\y\) we have
    \begin{equation}
        \begin{aligned}
            \E_{\x,\y}\big[ & \sharp\{\z=(z,u):u>t\text{ and }\x\rightsquigarrow\z\rightsquigarrow\y\}\big] \\ &= \int_{\R^d} \d z\int_t^1\d u \ \big(\tfrac{t}{u}\big)^{\Gamma}\rho\big(\tfrac{1}{\beta}t^\gamma u^{1-\gamma}|x-z|^d\big)\rho\big(\tfrac{1}{\beta}s^\gamma u^{1-\gamma}|y-z|^d\big) \\
            & \leq \int_{\R^d} \d z\int_t^1\d u \ \big(\tfrac{t}{u}\big)^{\Gamma}\rho\big(\tfrac{1}{\beta}t^\gamma u^{1-\gamma}|x-z|^d\big)\rho\big(\tfrac{1}{2^d\beta}s^\gamma u^{1-\gamma}|x-y|^d\big)\1_{\{|y-z|>\nicefrac{|x-y|}{2}\}} \\
            & \qquad + \int_{\R^d} \d z\int_t^1\d u \ \big(\tfrac{t}{u}\big)^{\Gamma}\rho\big(\tfrac{1}{2^d\beta}t^\gamma u^{1-\gamma}|x-y|^d\big)\rho\big(\tfrac{1}{\beta}s^\gamma u^{1-\gamma}|y-z|^d\big) \1_{\{|x-z|\geq \nicefrac{|x-y|}{2}\}} \label{eq:CalcTwoConnectSub1}
        \end{aligned}
    \end{equation}
    We bound the indicator function in the first integral of~\eqref{eq:CalcTwoConnectSub1} by \(1\) and infer by a simple change of variables
    \begin{equation*}
        \begin{aligned}
            2^{d\delta} \beta^{\delta+1}t^{\Gamma-\gamma} s^{-\gamma\delta}|x-y|^{-d\delta} \Big(\int_t^1 u^{-\Gamma-1+\gamma-\delta(1-\gamma)}\d u \Big)\Big(\int_{\R^d}\rho(|z|^d)\d z\Big) 
            &\leq \beta C \big(\tfrac{1}{\beta}s^\gamma t^{1-\gamma}|x-y|^d\big)^{-\delta} \\
            &= \beta C  \rho\big(\tfrac{1}{\beta}s^\gamma t^{1-\gamma}|x-y|^d\big).
        \end{aligned}
    \end{equation*}
    Here, we have used that \(\gamma<\nicefrac{(\delta+\Gamma)}{(\delta+1)}\) in the first inequality and our distance assumption in the last equation. Moreover, \(C\) is derived by the integration constants and reads \(C=\tfrac{2^{d\delta}\omega_d \delta}{d(\delta-1)(\delta(1-\gamma)+\gamma-\Gamma)}\). For the second integral of~\eqref{eq:CalcTwoConnectSub1} we calculate similarly
    \begin{equation*}
        \begin{aligned}
            \tfrac{2^{d\delta}\omega_d\delta}{d(\delta-1)}\beta^{\delta+1}t^{\Gamma-\gamma} s^{-\gamma}|x-y|^{-d\delta}\int_t^1 u^{-\Gamma-\delta(1-\gamma)+\gamma-1}\d u 
            & 
                \leq \beta C \big(\tfrac{1}{\beta}s^{\gamma/\delta}t^{1-\gamma\delta}|x-y|^d\big)^{-\delta} 
            \\& 
                \leq \beta C \rho\big(\tfrac{1}{\beta}s^{\gamma}t^{1-\gamma}|x-y|^d\big),
        \end{aligned}
    \end{equation*}
    as \(\gamma>\nicefrac{\gamma}{\delta}\) and with \(C\) as above. This proves~\eqref{eq:twoConnectSub}. The proof of~\eqref{eq:twoConnectSub2} works analogously.  

    We also start the proof of~(b) by calculating the expected number of vertices \(\y\) being part of the event \(\cE(\x)\). That is
    \begin{equation}\label{eq:CalcTwoConnectSup1}
        \begin{aligned}
            \int\limits_{|x-y|^d<t^{-\alpha_2}} & \d y\int_0^{t^{\alpha_1}}\d s\int\limits_{\R^d}\d z \int_{t}^1 \d u \ \big(\tfrac{t}{u}\big)^\Gamma\rho\big(\tfrac{1}{\beta}t^\gamma u^{1-\gamma}|x-z|^d\big)\rho\big(\tfrac{1}{\beta}s^{\gamma}u^{1-\gamma}|y-z|^d\big) \\
            & \geq t^{\Gamma}\int\limits_{|x-y|^d<t^{-\alpha_2}} \d y\int_0^{t^{\alpha_1}}\d s\int\limits_{|x-z|^d<t^{-\gamma}}\d z \int_t^1 \d u \ u^{-\Gamma} \rho\big(\tfrac{1}{\beta}\big)\rho\big(\tfrac{1}{\beta}s^{\gamma}|y-z|^d\big),
        \end{aligned}
    \end{equation} 
    Observe that from \(|x-z|^d<t^{-\gamma}<t^{-\alpha_2}\) and \(|x-y|^d<t^{-\alpha_2}\), we derive \(|y-z|^d<2 t^{-\alpha_2}\) by the triangle inequality. Using this and the fact that \(\Gamma<1\) as we would always have \(\gamma<\nicefrac{(\delta+\Gamma)}{(\delta+1)}\) and \(t<\nicefrac{1}{2}\) otherwise, the integral in~\eqref{eq:CalcTwoConnectSup1} reads
    \begin{equation*}
        \begin{aligned}
           (1-2^{\Gamma-1}) & t^{\Gamma} \int\limits_{|x-y|^d<t^{-\alpha_2}}  \d y\int_0^{t^{\alpha_1}}\d s\int\limits_{|x-z|^d<t^{-\gamma}}\d z \ \rho\big(\tfrac{1}{\beta}\big)\rho\big(\tfrac{1}{\beta}s^{\gamma}|y-z|^d\big) \\
            &\geq (1-2^{\Gamma-1}) t^{\Gamma}\rho(\nicefrac{1}{\beta}) t^{-\gamma}\int\limits_{|x-y|^d<t^{-\alpha_2}}\d y \int_0^{t^{\alpha_1}} \d s \ \rho\big(\tfrac{2}{\beta}t^{\gamma\alpha_1}t^{-\alpha_2}\big) \\
            &\geq (1-2^{\Gamma-1})\rho(\nicefrac{1}{\beta})\beta^{\delta}t^{\Gamma-\gamma-\delta(\gamma\alpha_1-\alpha_2)+\alpha_1-\alpha_2} \\
            &=c t^{\Gamma-\gamma + \alpha_2(\delta-1)-\alpha_1(\gamma\delta-1)}
        \end{aligned}
    \end{equation*}
    for a constant \(c=c(\beta)>0\). The proof finishes with the observation that, due to our choices of \(\alpha_1\) and \(\alpha_2\), we have \(\Gamma-\gamma+\alpha_2(\delta-1)-\alpha_1(\gamma\delta-1)<0\) and therefore
    \[
    	\P_{\x}(\cE(\x)) \geq 1-\exp(-c(\beta) t^{\Gamma-\gamma+\alpha_2(\delta-1)-\alpha_1(\gamma\delta-1)})\geq 1-\exp(-t^{-a}),
    \]
    for an appropriately adapted \(a\).
\end{proof}

The previous lemma tells us that there is a change in behaviour in the connection strategy when \(\gamma=\nicefrac{(\Gamma+\delta)}{(\delta+1)}\). In the next paragraph, we show that there exists a path to infinity for all \(\beta\). Let us comment before that \(\Gamma>\gamma\) implies \(\gamma<\nicefrac{(\gamma+\Gamma)}{(\delta+1)}\). Hence, there always exists a subcritical phase when the outdegree distribution is Poisson. In that situation there are simply too few reciprocal connections to make use of young connectors.

\paragraph{Absence of a subcritical weak percolation phase} 
\begin{proof}[Proof of Theorem~\ref{thm:perc} (ii).]
We use Lemma~\ref{lem:twoConnect} (b) to show that there exists a sequence \((\x_k=(x_k,s_k):k\in\N)\) of vertices such that \(s_k<s_{k-1}^{\alpha_1}\) and \(|x_k-x_{k-1}|^d<s_{k-1}^{-\alpha_2}\) such that \(\x_{k-1}\overset{2}{\underset{\x_{k.1},\x_k}{\rightsquigarrow}}\x_k\) with a positive probability. The almost sure existence of a directed paths to infinity then follows by ergodicity. Indeed, starting from a vertex \(\x_1=(x_1,s_1)\) with sufficiently small birth time \(s_1\), we have
\begin{equation*}
    \begin{aligned}
        \P_{\x_1}\{\text{there is no such sequence in \(\D\)}\} 
        & 
        	\leq \P_{\x_1}(\cE(\x_1)^c) +\sum_{k\geq 2} \P\Big(\cE(\x_k)^c \, \Big| \, \bigcap_{j=1}^{k-1} \cE(\x_j)\Big)  
        \\  & 
        	\leq \sum_{k\geq 1} \exp(-(s_1^{-a})^{k}) <1.
    \end{aligned}
\end{equation*}
This concludes the proof.
\end{proof}

\paragraph{Existence of a subcritical weak percolation phase}
To derive the existence of a weak percolation phase transition in \(\beta\), we replicate some arguemtns of~\cite{GLM2021}. There, the authors bound the probability existence of a \emph{shortcut-free} path of length \(n\) starting in the root in the undirected model. Transferred to our setting, a path \(\mathscr{P}=(\x_0,\dots,\x_n)\) is called \emph{shortcut-free}, if \(\mathscr{N}^\text{in}(\x_{j})\cap \mathscr{P} = \x_{j-1}\) and \(\mathscr{N}^\text{out}(\x_j)\cap \mathscr{P} = \x_{j+1}\) for all \(j=1,\dots,n-1\).  Note that there exists a directed shortcut-free path to infinity whenever there exists a directed path to infinity. We shortly recap the main arguments of~\cite{GLM2021} and explain how they are applicable to our setting. Throughout the paragraph, we assume that \(\gamma<\nicefrac{(\delta+\Gamma)}{(\delta+1)}\).

To make use of the shortcut-free property, we work in the following with the underlying undirected graph \(\G=\D[\beta,\gamma,\delta,1]\), and recall the profile function \(\rho(x)=1\wedge x^{-\delta}\). From this underlying graph, we construct the directed graph \(\widetilde{\D}\) as above but with the reciprocity profile \(\pi\), replaced with 
\begin{equation*} %\label{eq:refollowCoupling}
    \widetilde{\pi}(s,t,|x-y|^d) = \1_{\{|x-y|^d\leq \beta s^{-\gamma}t^{\gamma-1}\}}+\left(\frac{s}{t}\right)^\Gamma \1_{\{|x-y|^d>\beta s^{-\gamma}t^{\gamma-1}\}}, \qquad \text{ for }s<t.
\end{equation*}
That is, in \(\widetilde{\D}\), all vertices \(\x=(x,t)\) and \(\y=(y,s)\) with \(|x-y|^d\leq \beta (s\wedge t)^{-\gamma}(s\vee t)^{\gamma-1}\), are connected by a double-arc almost surely. As a result
\[
    \P_{o}\{\0\rightarrow\infty \text{ in }\D\}\leq \P_o\{\0\rightarrow\infty \text{ in }\widetilde{\D}\}.
\]
and we work in the following explicitly on the digraph \(\widetilde{\D}\).

\begin{figure}
\begin{center}
			\begin{tikzpicture}[scale=0.3, every node/.style={scale=0.3}]
				\node (Z) at (-2.5,8.5)[circle, draw,scale=1.5]{1};
				\draw[->] (-1,-0.5) -- (-1, 8)
					node[left,scale=2] {$t$};
				\node[->] (A) at (0,5)[circle, fill=black, label ={}]{};	
    			\node[->] (B) at (2.5,3)[circle, fill = black, label={}] {};
 %   			\node (C) at (5,4) [circle, draw, label={}] {};
    			\node[->] (D) at (7.5,7)[circle, draw, label = {}] {};
    			\node[->] (E) at (10, 6)[circle, draw, label={}] {};
    			\node[->] (F) at (12.5,1) [circle, fill= black, label={} ] {};
    			\node[->] (G) at (15,4.5) [circle, draw, label={}]{};
    			\node[->] (H) at (17.5,2.5)[circle, fill=black, label={}]{};
    			\draw[->] (A) to (B);
				\draw[->] (B) to (D);
%				\draw (C) to (D);
				\draw[->] (D) to (E);
				\draw[->] (E) to (F);
				\draw[->] (F) to (G);
				\draw[->] (G) to (H);
			\end{tikzpicture}
			\hspace{1 cm}
			\begin{tikzpicture}[scale=0.3, every node/.style={scale=0.3}]
				\node (Z) at (-2.5,8.5)[circle, draw,scale=1.5]{2};
				\draw[->] (-1,-0.5) -- (-1, 8)
					node[left,scale=2] {$t$};
				\node (A) at (0,5)[circle, fill=black, label ={}]{};
    			\node (B) at (2.5,3)[circle, fill = black, label={}] {};
%    			\node (C) at (5,4) [circle, draw, label={}] {};
    			\node (D) at (7.5,7)[circle, draw, label = {}, dotted] {};
    			\node (E) at (10, 6)[circle, draw, label={}] {};
    			\node (F) at (12.5,1) [circle, fill= black, label={} ] {};
    			\node (G) at (15,4.5) [circle, draw, label={}]{};
    			\node (H) at (17.5,2.5)[circle, fill=black, label={}]{};
    			\draw[->] (A) to (B);
%				\draw (B) to (C);
				\draw[dotted,->] (B) to (D);
				\draw[dotted,->] (D) to (E);
				\draw[->] (B) to (E);
				\draw[->] (E) to (F);
				\draw[->] (F) to (G);
				\draw[->] (G) to (H);
%				\draw (F) to (H);
			\end{tikzpicture}
			\begin{tikzpicture}[scale=0.3, every node/.style={scale=0.3}]
				\node (Z) at (-2.5,8.5)[circle, draw,scale=1.5]{3};
				\draw[->] (-1,-0.5) -- (-1, 8)
					node[left,scale=2] {$t$};
				\node (A) at (0,5)[circle, fill=black, label ={}]{};
    			\node (B) at (2.5,3)[circle, fill = black, label={}] {};
  %  			\node (C) at (5,4) [circle, draw, label={}] {};
    			\node (D) at (7.5,7)[label = {}] {};
    			\node (E) at (10, 6)[circle, draw, label={}, dotted] {};
    			\node (F) at (12.5,1) [circle, fill= black, label={} ] {};
    			\node (G) at (15,4.5) [circle, draw, label ={}]{};
    			\node (H) at (17.5,2.5)[circle, fill=black, label={}]{};
    			\draw[->] (A) to (B);
				\draw[->] (B) to (F);
				\draw[dotted,->] (B) to (E);
				\draw[dotted,->] (E) to (F);
				\draw[->] (B) to (F);
				\draw[->] (F) to (G);
				\draw[->] (G) to (H);
			\end{tikzpicture}
			\hspace{1 cm}
			\begin{tikzpicture}[scale=0.3, every node/.style={scale=0.3}]
				\node (Z) at (-2.5,8.5)[circle, draw,scale=1.5]{4};
				\draw[->] (-1,-0.5) -- (-1, 8)
					node[left,scale=2] {$t$};
				\node (A) at (0,5)[circle, fill=black, label ={}]{};
    			\node (B) at (2.5,3)[circle, fill = black, label={}] {};
 %   			\node (C) at (5,4) [circle, draw, label={}, dotted] {};
    			\node (D) at (7.5,7)[label = {}] {};
    			\node (E) at (10, 6)[label={}] {};
    			\node (F) at (12.5,1) [circle, fill= black, label={} ] {};
    			\node (G) at (15,4.5) [circle, draw, dotted, label={}]{};
    			\node (H) at (17.5,2.5)[circle, fill=black, label={}]{};
    			\draw[->] (A) to (B);
				\draw[dotted,->] (F) to (G);
				\draw[dotted,->] (G) to (H);
				\draw[->] (B) to (F);
				\draw[->] (F) to (H);
			\end{tikzpicture}
			\caption{A directed path where a vertex's birth time is denoted on the $t$-axis. The vertices of the skeleton with running minimum birth time are in black. We successively remove all local maxima, starting with the youngest, and replace them by directed edges until the directed path, only containing the skeleton vertices, is left.}
	\label{fig:Skeleton}
\end{center}
\end{figure}
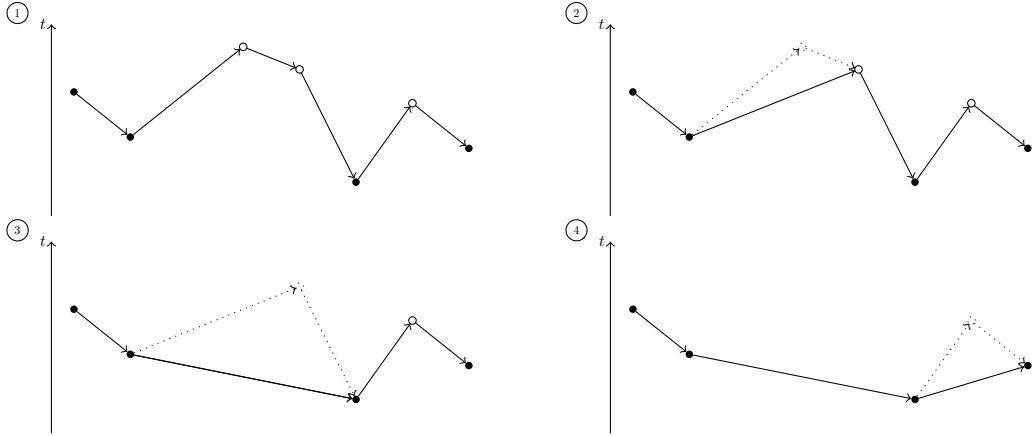

For a path \(\mathscr{P}\) we call the collection of vertices with running minimum birth time from both sides the \emph{skeleton} of \(\mathscr{P}\). That is, we start from the root vertex \((x_0,t_0)\) and search for the first vertex \((x_{j_1},t_{j_1})\) that has mark \(t_{j_1}<t_0\). Restarting from this vertex,  we search for the next vertex with smaller birth time until we reach the vertex with smallest birth time of the path. Afterwards we do the same but starting from the last vertex of the path \((x_n,t_n)\) and going backwards across the indices. Another possibility to identify the path's skeleton is the following: We call a vertex \(\x_j\in \mathscr{P}\setminus\{\x_{0},\x_n\}\) local maximum if \(t_j>t_{j-1}\) and \(t_j>t_{j+1}\). Put differently, \(\x_j\) is younger than its preceding and subsequent vertex. We now successively remove all local maxima from \(\mathscr{P}\) as follows: First, take the local maximum in \(\mathscr{P}\) with the greatest birth time, remove it from \(\mathscr{P}\) and connect its former neighbours by a directed edge oriented from preceding to subsequent vertex. In the resulting path, we take the local maximum of greatest birth time and remove it, repeating until there is no local maximum left, see Figure~\ref{fig:Skeleton}. 

The idea is now the following: To bound the probability that the root \(\0\) starts a directed shortcut-free paths of length \(n\) in \(\widetilde{\D}\), we first condition on the skeleton of the paths. Afterwards, we remove step by step the local maxima of \(\mathscr{P}\) and replace them with directed edges. Since \(\gamma<\nicefrac{(\delta+\Gamma)}{(\delta+1)}\), we have by Lemma~\ref{lem:twoConnect}~(a) that the probability of the new arc is up to a \(\beta\) dependent constant larger than the probability of the previous connection via the intermediate vertex. Here, it is important to note that all vertices \(\x\) and \(\y\) within distance \(|x-y|^d<\beta (t_x\wedge t_y)^{-\gamma}(t_y\vee t_y)^{\gamma-1}\) are already connected by a double-arc in \(\widetilde{\D}\). Hence, in this situation we never use connectors due to the shortcut-free property. This allows us to argue as in the proof of~\cite[Lemma~2.3]{GLM2021} to infer for all \(k\geq 1\)
\begin{equation*}% \label{eq:kConnection}
    \P_{\x,\y}\Big\{\x\overset{k}{\underset{\x,\y}{\rightsquigarrow}}\y \text{ in }\widetilde{\D}\Big\} \leq (4\beta C)^{k-1} \P_{\x,\y}\{\x\rightarrow\y \text{ in }\widetilde{\D}\},
\end{equation*}
where \(\x\overset{k}{\underset{\x,\y}{\rightsquigarrow}}\y\) denotes the event that \(\x\) is connected to \(\y\) by a directed path of length \(k\), in which all intermediate connectors are younger than \(\x\) and \(\y\) themselves. Note that no assumption whether \(\x\) or \(\y\) is the older vertex has been made. 

Let now \(\mathscr{S}=(\x_0,\x_1,\dots,\x_k)\) be a given skeleton and let us write \(\x\overset{n}{\underset{\mathscr{S}}{\rightsquigarrow}}\y\) for the event that there is a directed path of length \(n\) from \(\x_0=\x\) to \(\x_k=\y\) with skeleton \(\mathscr{S}\). We can then use the BK-inequality~\cite{BK85} in a version of~\cite{HvdHLM20} as outlined in~\cite[Eq.\ (11)]{GLM2021}
to deduce
\begin{equation}\label{eq:BK}
    \begin{aligned}
        \P_{\x_0,\x_1,\dots,\x_k}\big\{\x\overset{n}{\underset{\mathscr{S}}{\rightsquigarrow}}\y \text{ in }\widetilde{\D}\big\} 
        & 
        	\leq \sum_{\substack{n_1,\dots,n_k\in\N \\ n_1+\dots +n_k=n}} \, \prod_{j=1}^k \P_{\x_{j-1},\x_j}\{\x_{j-1}\overset{n_j}{\underset{\x_{j-1},\x_j}{\rightsquigarrow}}\x_j \text{ in }\widetilde{\D}\} 
        \\ &
        	\leq (4\beta C)^{n-k}\binom{n-1}{k-1} \prod_{j=1}^k  \P_{\x_{j-1},\x_j}\{\x_{j-1}\to\x_j \text{ in }\widetilde{\D}\}.
    \end{aligned}
\end{equation}
From here, the proof finishes by applying the last term of~\eqref{eq:BK} to paths with decreasing skeleton and performing a expectation bound. 

\begin{proof}[Proof of Theorem~\ref{thm:perc}~(i).]
 First observe that, on the event \(\{\o\rightsquigarrow\infty\}\), we have \(\inf\{t_x\colon \o\rightsquigarrow\x\}=0\) since, by ergodicity and monotonicity of \(\{\o\rightsquigarrow\infty\}\) under decreasing birth times, the infinite weak component cannot avoid entire birth time intervals that have positive mass. Consequently, there must exists a directed path \(\mathscr{P}\) from \(\o\) that contains infinitely many subpaths that start in \(\o\) and end in their oldest vertex. By construction such a subpaths has a skeleton that is decreasing in its vertices birth times. Denote by \(\mathscr{S}_k=\{\x_0=\o,\x_1,\dots,\x_k\}\) such a skeleton of length \(k\) with \(t_0>t_1>\dots>t_k\), then by~\eqref{eq:BK} and Mecke's equation, we have
 \begin{equation*}
 	\begin{aligned}
 		\sum_{n=1}^\infty &
 			 \P_o\{\exists \text{ directed outgoing path with decreasing skeleton of length } n \text{ starting in }\o\}
 		\\ & \hspace{-0.3 cm}
 			\leq \sum_{n=1}^\infty \sum_{k=1}^{n} \P_o\big\{\exists \mathscr{S}_k\colon \o\overset{n}{\underset{\mathscr{S}_k}{\rightsquigarrow}}\x_k\big \} 
 		%\\ & \hspace{-0.3 cm}
 		%	\leq \sum_{n=1}^\infty \sum_{k=1}^{n} (4\beta C)^{n-k}\binom{n-1}{k-1} \int\limits_{(\R^d)^k} \d\x_1\cdots \d\x_k \hspace{-0.5 cm}\int\limits_{1>t_0>\dots>t_k>0} \hspace{-0.5 cm} \d t_0\cdots \d t_k \,  \prod_{j=1}^k  \P_{\x_{j-1},\x_j}\{\x_{j-1}\to\x_j \text{ in }\widetilde{\D}\}
 		\\ & \hspace{-0.3 cm}
 			\leq \sum_{n=1}^\infty \sum_{k=1}^{n} (4\beta C)^{n-k}\binom{n-1}{k-1} \int\limits_{(\R^d)^k} \d x_1\cdots \d x_k \hspace{-0.5 cm}\int\limits_{1>t_0>\dots>t_k>0} \hspace{-0.5 cm} \d t_0\cdots \d t_k \,  \prod_{j=1}^k \rho\big(\beta^{-1}t_{j-1}^{1-\gamma}t_j^\gamma |x_{j}-x_{j-1}|^d\big). 
 	\end{aligned}
 \end{equation*}
Note that we have used the assumption \(\gamma<\nicefrac{(\delta+\Gamma)}{(\delta+1)}\) here, as~\eqref{eq:BK} builds on Lemma~\ref{lem:twoConnect}~(a). Now, starting with \(j=k\) and and going backwards across the indices, we successively perform the change of variables \(z_j = (\beta^{-1}t_{j-1}^{1-\gamma}t_j^{\gamma})^{1/d}(x_j-x_{j-1})\) and obtain
\[
	\begin{aligned}
		\int\limits_{(\R^d)^k} 
		&
			\d x_1\cdots \d x_k \hspace{-0.5 cm}\int\limits_{1>t_0>\dots>t_k>0} \hspace{-0.5 cm} \d t_0\cdots \d t_k \,  \prod_{j=1}^k \rho\big(\beta^{-1}t_{j-1}^{1-\gamma}t_j^\gamma |x_{j}-x_{j-1}|^d\big)
		\\ &
			= \Big(\beta \int\limits_{\R^d} \d z \, \rho(|z|^d)\Big)^k \int\limits_0^1 \d t_0 \int\limits_0^{t_0}\d t_1 \dots \int\limits_0^{t_{k-1}} \d t_k \, t_0^{\gamma-1}\big(\prod_{j=1}^{k-1} t_j^{-1}\big) t_k^{-\gamma}
	%	\\ &
			\leq \Big(\beta \frac{\omega_d \delta}{\delta-1} \frac{1}{1-\gamma}\Big)^k. 
	\end{aligned}
\] 	
Plugging this back in the previous calculation, we infer the existence of a constant \(C'>4\) such that
\[
	\sum_{n\in\N }\P_o\{\exists \text{ directed path with decreasing skeleton of length } n \text{ starting in }\o\} \leq \sum_{n\in\N} (\beta C')^n.
\]
Choosing \(\beta<1/C'\), the Borel-Cantelli Lemma yields that only finitely many paths starting in the origin that end in their oldest vertex exist, almost surely. Therefore, \(\inf\{t_x\colon \o\rightsquigarrow\x\}>0\), which ultimately implies \(\P_o(\o\rightsquigarrow\infty)=0\), as required. 
\end{proof}

\paragraph{Acknowledgement.} LL received support by the Leibniz Association within the Leibniz Junior Research Group on \textit{Probabilistic Methods for Dynamic Communication Networks} as part of the Leibniz Competition (grant no.\ J105/2020) and by the Deutsche Forschungsgemeinschaft (DFG, German Research Foundation) under Germany's Excellence Strategy - The Berlin Mathematics Research Center MATH+ (EXC-2046/1, EXC-2046/2, project ID: 390685689) through the project \emph{Information flow \& emergent behavior in complex networks}. CM's research was partly funded by Deutsche Forschungsgemeinschaft (DFG, German Research Foundation) – SPP 2265 443916008.

%%%%%%%%%%%%%%%%%%
%%% References %%%
%%%%%%%%%%%%%%%%%%
\section*{References}
\renewcommand*{\bibfont}{\footnotesize}
\emergencystretch=1em
\printbibliography[heading = none]

\appendix

\section{Expected number of bow-ties}\label{sec:int_results}
\begin{lemma}\label{lem:orderSopen} Consider \(S(t_o)\) as defined in~\eqref{eq:defS}. Then,
\begin{equation}\label{eq:lemSopenAsymp}
\int_{1/(t\log t)}^1 S(t_o)\d t_o \asymp    
	\begin{cases}
    	1 & \text{if } \gamma < \tfrac{1}{2},
		\\
		(\log t)^{1+\1{\{\Gamma=0\}}} & \text{if } \gamma = \tfrac{1}{2},
		\\
        (t\log t)^{2\gamma-1}(\log t)^{\1\{\Gamma=0\}} & \text{if } \gamma > \tfrac{1}{2}.
	\end{cases}
\end{equation}
In particular, if \(\gamma<1/2\), then \(\int_{1/(t\log t)}^1 S(t_o) \d t_o\) is uniformly bounded in \(t\).
\end{lemma}
\begin{proof}
Using Mecke's equation, we can write
\begin{equation}\label{eq:uIntegral}
    S(t_o)
    = \hspace{-0.5cm}
    \int\limits_{(\mathbb{T}_t^d\times(\nicefrac{1}{t\log t},1))^3}\hspace{-0.5 cm}
    \P_{\o,\u}(\0\to\mathbf{u})\;\P_{\o,\v}(\0\to\mathbf{v})\;\P_{\w,\v}(\mathbf{w}\to\mathbf{v})
    \d\mathbf{u}\d\mathbf{v}\d\mathbf{w},
\end{equation}
where $d\mathbf{u}=\d u \d t_u$ etc. The three arc indicators are conditionally independent given all positions
and birth times, because we never simultaneously require both $\x\to\y$ and $\y\to\x$ for any pair.
The occupation probability for a generic arc $\x=(x,t_x)\to\y=(y,t_y)$ can be written as
\[
    \P_{\x,\y}(\x\to\y) =
    \rho\Big(\beta^{-1} a(t_x,t_y)\,\textrm{d}_t(x,y)^d\Big)\,\Pi(t_x,t_y),
\]
where
\[
a(t_x,t_y)=
\begin{cases}
t_y^\gamma t_x^{1-\gamma}, & t_x>t_y, 
\\
t_x^\gamma\,t_y^{1-\gamma}, & t_x<t_y,
\end{cases}
\qquad
\Pi(t_x,t_y)=
\begin{cases}
1, & t_x>t_y \quad(\text{forward}),
\\
(t_x/t_y)^\Gamma, & t_x<t_y \quad(\text{reciprocal}).
\end{cases}
\]
Note \(\d_t(o,u)=|u|\) so that a change of variables yields, for any $a>0$,
\begin{equation} \label{eq:spaceint-open}
    \begin{aligned}
        \int\limits_{\bbT_t^d} \rho(\beta^{-1}a \d_t(o,x)^d) \d x
        & 
            = \!\!\!\! \int\limits_{\big[-\tfrac{t^{1/d}}{2}, \tfrac{t^{1/d}}{2}\big]^d} \!\!\!\! \rho(\beta^{-1}a |x|^d) \d x = \frac{\beta}{a} \!\!\!\!\! \int\limits_{\big[-\tfrac{(at/\beta)^{1/d}}{2}, \tfrac{(at\beta)^{1/d}}{2}\big]^d} \!\!\!\!\!\! \rho(|x|^d) \d x \asymp a^{-1},
    \end{aligned}
\end{equation}
by the integrability assumption on \(\rho\), that is 
\[
    1\leq \sup_{t} \!\!\! \int\limits_{[-t^{1/d}/2,t^{1/d}/2]^d}  \!\!\! \rho(|x|^d) \d x =  \int\limits_{\R^d} \rho(|x|^d)\d x=:c_\rho <\infty.
\]
Define the \emph{spatially integrated connection kernel}
\[
\kappa_t(t_x,t_y)
:=
\int_{\mathbb{T}_t^d} \P_{\o,\z}\big((0,t_x)\to(z,t_y)\big)\d z
\]
and obtain
\begin{equation}\label{eq:kappa-open}
\kappa_t(t_x,t_y) \asymp \frac{\Pi(t_x,t_y)}{a(t_x,t_y)} \asymp 
\1_{\{t_y<t_x\}}t_y^{-\gamma} t_x^{\gamma-1}
+
\1_{\{t_x<t_y\}}\,
t_x^{\Gamma-\gamma}\,t_y^{-(1-\gamma+\Gamma)},
\end{equation}
noting that the constants hidden in \(\asymp\) do not depend on \(t,\), \(t_x\), or \(t_y\). Since the $\mathbf{u}$-integral in~\eqref{eq:uIntegral} does not involve $\mathbf{v}$ or $\mathbf{w}$, we factorise
\[
    S(t_o)=A(t_o)\,B(t_o),
\]
where
\[
A(t_o)
:=
\int\limits_{1/(t\log t)}^1 \kappa_t(t_o,t_u)\d t_u,
\qquad
B(t_o)
:=
\int\limits_{1/(t\log t)}^1 \int\limits_{1/(t\log t)}^1
\kappa_t(t_o,t_v)\,\kappa_t(t_w,t_v)\d t_w\d t_v.
\]
In particular, $A(t_o)$ is the expected outdegree (under birth-time truncation) of a vertex with mark $t_o$. Hence, by performing the same calculations as in the proof of Theorem~\ref{thm:neighbourhood}, cf.~\eqref{eq:degCalc1} and~\eqref{eq:degCalc2},
\[
    A(t_o)\asymp t_o^{\Gamma-\gamma}\1\{\Gamma\in [0,\gamma)\}+\log(1/t_o)\1\{\Gamma=\gamma\} + \1\{\Gamma>\gamma\},
\]
noting that the truncation at \(1/(t\log t)\) only affects lower order terms.
%
% \paragraph{Asymptotics of $\boldsymbol{A(t_o)}$.}
% Splitting at $t_u=t_0$ and using \eqref{eq:kappa-open},
% the forward part contributes $\Theta(1)$ (uniformly in $t_o\in[1/(t\log t),1]$), while the reciprocal part gives
% \[
% \int_{t_0}^1 t_o^{\Gamma-\gamma}\,t_u^{-(1-\gamma+\Gamma)}\d t_u
% \asymp
% \begin{cases}
% t_o^{\Gamma-\gamma}, & \Gamma\in[0,\gamma),\\
% \log(1/t_o), & \Gamma=\gamma,\\
% 1, & \Gamma>\gamma.
% \end{cases}
% \]
% Hence, as $t_o\to0$ above the lower envelope $1/(t\log t)$,
% \[
% A(t_o)\asymp t_o^{\Gamma-\gamma}\1\{\Gamma\in [0,\gamma)\}+\log(1/t_o)\1\{\Gamma=\gamma\} + \1\{\Gamma>\gamma\}.
% \]
To obtain the order of \(B(t_o)\), note that the two spatial integrals decouple as the $\mathbf{w}$-integral over space depends on $\d_t(w,v)$,
but after integration over $w\in\mathbb{T}_t^d$ the result is independent of the position $v$ by translation invariance.
Thus, we may write
\begin{equation}\label{eq:defB(to)}
	B(t_o)=\int_{1/(t\log t)}^1 \kappa_t(t_o,t_v)\,h_t(t_v)\d t_v, \qquad h_t(t_v):=\int_{1/(t\log t)}^1 \kappa_t(t_w,t_v)\d t_w .
\end{equation}
Splitting according to whether $\mathbf{w}$ is younger or older than $\mathbf{v}$ and using \eqref{eq:kappa-open},
\[
	\begin{aligned}
		h_t(t_v)
		&
			\asymp\int_{t_v}^{1}{t_v^{-\gamma} t_w^{\gamma-1}} \d t_w
			+\int_{1/(t\log t)}^{t_v} t_w^{\Gamma-\gamma}\,t_v^{-(1-\gamma+\Gamma)} \d t_w
		%\\ &
		%	\asymp \frac{t_v^{-\gamma}-1}{\gamma}+1
        %\\ &
			\asymp t_v^{-\gamma},
	\end{aligned}
\]
implying
\[
B(t_o)\asymp \int_{1/(t\log t)}^{1} \kappa_t(t_o,t_v)\;t_v^{-\gamma}\d t_v
= I_1(t_o)+I_2(t_o),
\]
where
\[
\begin{aligned}
    I_1(t_o)
    &= \int\limits_{1/(t\log t)}^{t_o}\Theta\Big({t_v^{-\gamma} t_o^{\gamma-1}}\Big)\,t_v^{-\gamma}\d t_v
     \asymp
    {t_o^{\gamma-1}}\int\limits_{1/(t\log t)}^{t_o}t_v^{-2\gamma}\d t_v
    &&(\text{forward: }t_v<t_o),\\
    I_2(t_o)
    &= \int_{t_o}^{1}\Theta\big(t_o^{\Gamma-\gamma}\,t_v^{-(1-\gamma+\Gamma)}\big)\,t_v^{-\gamma}\d t_v
    \asymp
    t_o^{\Gamma-\gamma}\int_{t_o}^{1}t_v^{-(1+\Gamma)}\d t_v
    &&(\text{reciprocal: }t_v>t_o).
\end{aligned}
\]
Straight-forward integration yields 
\[
    I_1(t_o) 
    \asymp
    \begin{cases}
        t_o^{-\gamma}, & \text{ if } \gamma<1/2,
        \\
        t_o^{-1/2}\log(t_o t), & \text{ if } \gamma=1/2,
        \\
        t_o^{\gamma-1}(t\log t)^{2\gamma-1}, & \text{ if } \gamma>1/2.
    \end{cases}
\]
Similarly,
\[
    I_2(t_o)
    \asymp
    \begin{cases}
        t_o^{-\gamma}\log(1/t_o), & \text{ if } \Gamma=0,
        \\
        t_o^{-\gamma}, & \text{ if } \Gamma>0.
    \end{cases}
\]
Consequently,
\begin{equation} \label{eq:orderB(to)}
	\begin{aligned}
	B(t_o) 
	&
	=
	\begin{cases}
		\Theta(t_o^{-\gamma}(\log(1/t_o))^{\1\{\Gamma=0\}}), & \text{ if }\gamma<1/2,
		 \\
		 \Theta(t_o^{-1/2}\log(t_o t)))+\Theta(t_o^{-1/2}(\log(1/t_o))^{\1\{\Gamma=0\}}),& \text{ if }\gamma=1/2,
		 \\ 
		 \Theta(t_o^{\gamma-1}(t\log t)^{2\gamma-1}) + \Theta(t_o^{-\gamma}(\log(1/t_o))^{\1\{\Gamma=0\}}), & \text{ if }\gamma>1/2.
	\end{cases}
	\\ &
	\asymp
	\begin{cases}
		t_o^{-\gamma}(\log(1/t_o))^{\1\{\Gamma=0\}}, & \text{ if }\gamma<1/2,
		 \\
		 t_o^{-1/2}\log(t_o t), \ \phantom{\Theta ()+\Theta(t_o^{-1/2}(\log(1/t_o))^{\1\{\Gamma=0\}}} & \text{ if }\gamma=1/2,
		 \\ 
		 t_o^{\gamma-1}(t\log t)^{2\gamma-1}, & \text{ if }\gamma>1/2.
	\end{cases}
	\end{aligned}
\end{equation}
To finish the proof, it remains to combine the derived asymptotics to obtain the main order of
\begin{equation*}
\int_{1/(t\log t)}^1 S(t_o)\d t_o
\;\asymp\;
\int_{1/(t\log t)}^1 B(t_o)\ A(t_o)\d t_o.
\end{equation*}

\begin{description}
	\item[Case $\boldsymbol{\gamma<1/2}$.]
		The product $B(t_o)A(t_o)$ is integrable near zero for every $\Gamma\ge0$. For $\Gamma>\gamma$, the integrand is $t_o^{-\gamma}$. For $\Gamma=\gamma$, it is $t_o^{-\gamma}\log(1/t_o)$, and for $0\leq \Gamma<\gamma$, it is $t_o^{\Gamma-2\gamma}(\log(1/t_o))^{\1\{\Gamma=0\}}$, which is integrable since \(\Gamma-2\gamma\geq -2\gamma>-1\). Therefore, integrating with respect to \(t_o\) yields \(\int_{1/(t\log t)}^1 B(t_o)\ A(t_o)\d t_o=\Theta(1)\).
	\item[Case $\boldsymbol{\gamma=1/2}$.] 
        For $\Gamma>1/2$, the product is of order \(\Theta(t_o^{-1/2}\log t)\), which integrates (with respect to \(t_o\) to \(\Theta(\log t)\). The same applies for \(\Gamma=\gamma=1/2\) as the additional \(\log(1/t_o)\) term does not affect integrability. For $0<\Gamma<1/2$, we get $\log(t)\int_{1/(t\log t)}^1 t_o^{\Gamma-1}\d t_o = \Theta(\log t)$. For $\Gamma=0$, the integral is $\log(t)\int_{1/(t\log t)}^1 t_o^{-1}\d t_o = \Theta(\log(t)^2)$.
	\item[Case $\boldsymbol{\gamma>1/2}$.] 
        For $\Gamma\geq \gamma$, the integral becomes $\Theta((t\log t)^{2\gamma-1})$, since \(t_o^{\gamma-1}\log(1/t_o)\) is integrable. For $0\leq \Gamma<\gamma$, we get $t^{2\gamma-1}\int_{1/(t\log t)}^1 t_o^{\Gamma-1}\d t_o = \Theta(t^{2\gamma-1}(\log t)^{\1\{\Gamma=0\}})$.    \end{description} 
Summarising the above calculations yields~\eqref{eq:lemSopenAsymp}. The uniform boundedness in \(t\) in case \(\gamma<1/2\) is a result of~\eqref{eq:lemSopenAsymp} together with the observation that the spatial integration terms absorbed in the $\Theta$ notation are uniformly bounded by \(c_\rho\).
\end{proof}

\begin{lemma}\label{lem:orderSclosed}
Assume \(\gamma\in[1/2,1)\) and consider \(S^c(t_o)\) as defined in~\eqref{eq:defSclosed}. 
\begin{enumerate}[(i)]
	\item 
		If \(\Gamma>0\), we have
		\begin{equation*}
			\int\limits_{1/(t\log t)}^1 S^c(t_o)\d t_o
			=
			\begin{cases}
				O(1), & \gamma<2/3,
				\\
				O(\log t), & \gamma=2/3,
				\\
				O((t\log t)^{3\gamma-2}), & 2/3<\gamma<1.
			\end{cases}
		\end{equation*}
	\item
		If \(\Gamma=0\), we have instead
		\[
			\int\limits_{1/(t\log t)}^1 S^c(t_o)\d t_o
			=
			\begin{cases}
				O(1), &  \gamma<2/3,
				\\
				O(\log t), & \gamma =2/3
				\\
				O((t\log t)^{3-2/\gamma}), & \gamma>2/3.
			\end{cases}
		\]
		In particular, for all \(\gamma>2/3\), we have \(\int_{1/(t\log t)}^1 S^c(t_o)\d t_o=o((t\log t)^{2\gamma-1})\).
\end{enumerate}
    
\end{lemma}
\begin{proof}
We separately treat the cases \(\Gamma=0\) and \(\Gamma>0\).
\paragraph{The case \(\boldsymbol{\Gamma>0}\).}
Recall the spatial integration kernel \(\kappa_t\) and its asymptotics~\eqref{eq:kappa-open} from the previous proof. Further recall that, during the course of the proof, all birth times are larger than \(1/(t\log t)\) and recall the quantity under consideration
\[
	S^c(t_o)= \int\limits_{(\bbT_t^d\times(1/t,1))^3} \!\!\! \P_{\o,\u}(\0\to\u)\,\P_{\o,\v}(\0\to\v)\,\P_{\w,\u}(\w\to\u)\,\P_{\w,\v}(\w\to\v)\d \u \d \v\d \w.
\]
To keep notation concise, we abbreviate \(\bar t=1/(t\log t)\) in the following. Write, for fixed locations $(u,v)$ and birth times $(t_u,t_v,t_w)$,
\[
K_t(u,v;t_u,t_v,t_w)
:=
\int_{\bbT_t^d} \P_{\w,\u}(\w\to\u)\,\P_{\w,\v}(\w\to\v)\d w.
\]
Clearly, $\P_{\w,\u}(\w\to\u)\P_{\w,\v}(\w\to\v)\le \P_{\w,\u}(\w\to\u)$ as well as
$\P_{\w,\u}(\w\to\u)\P_{\w,\v}(\w\to\v)\le \P_{\w,\v}(\w\to\v)$, hence
\begin{equation}\label{eq:Sc_upperbound}
    \begin{aligned}
        K_t(u,v;t_u,t_v,t_w)
        &
            \le \min\!\Big\{\int_{\bbT_t^d}\P_{\w,\u}(\w\to\u)\d w,\;\int_{\bbT_t^d}\P_{\w,\v}(\w\to\v)\d w\Big\}
        \\ &
            =\min\big\{\kappa_t(t_w,t_u),\kappa_t(t_w,t_v)\big\}.
    \end{aligned}
\end{equation}
Moreover, for each fixed $t_w$, the map $s\mapsto \kappa_t(t_w,s)$ is non-increasing on $(\bar t,1)$
(this is immediate from \eqref{eq:kappa-open}), so after symmetrising over $t_u,t_v$, we may assume
$t_u\le t_v$ and obtain
\[
\min\big\{\kappa_t(t_w,t_u),\kappa_t(t_w,t_v)\big\}
=
\kappa_t(t_w,t_v).
\]
Using translation invariance just like in the previous proof (so that the spatial integrals factorise), we therefore get the bound
\begin{equation}\label{eq:Sc-master}
	\begin{aligned}
		S^c(t_o) 
            &
                \leq 2\int\limits_{\bar t}^1 \! \d t_v \int\limits_{\bar t}^{t_v} \! \d t_u \int\limits_{\bar t}^1 \! \d t_w \Big[\int\limits_{\bbT_t^d}\! \d v\Big(\P_{\o,\v}(\o\to\v) \! \int\limits_{\bbT_t^d}\! \d w \ \P_{\w,\v}(\w\to\v)\Big)\! \int\limits_{\bbT_t^d}\! \d u\ \P_{\o,\u}(\o\to\u)\Big]
		\\ &  
    		= 2\int\limits_{\bar t}^1 \d t_v \ \kappa_t(t_o,t_v)\,U_t(t_o,t_v)\,h_t(t_v),
\end{aligned}
\end{equation}
where 
\[
U_t(t_o,s):=\int\limits_{\bar t}^{s}\kappa_t(t_o,r)\d r
\quad \text{ and }\quad 
h_t(s):=\int\limits_{\bar t}^{1}\kappa_t(t_w,s)\d t_w.
\]
We now derive the orders of the quantities that build \(S^c(t_o)\). First, one checks, just like in the open bow-tie case, from \eqref{eq:kappa-open} that
\begin{equation}\label{eq:ht-bound}
    h_t(s)=\Theta(s^{-\gamma}), \qquad \text{uniformly for } s\in[\bar t,1],
\end{equation}
since the forward part $t_w>s$ contributes $\Theta(s^{-\gamma})$ and the reciprocal part $t_w<s$ is $\Theta(1)$. To obtain a bound on \(U_t\), we start with the case $s\le t_o$, and observe that $\0\to (\,\cdot\,,r)$ is forward for all $r\le s$ so that \eqref{eq:kappa-open} gives
\begin{equation}\label{eq:Ut-small-s}
    U_t(t_o,s) \asymp t_o^{\gamma-1} \!\!\! \int\limits_{1/(t\log t)}^{s} \!\!\! r^{-\gamma}\d r \asymp t_o^{\gamma-1}\,s^{1-\gamma},
\qquad s\le t_o.
\end{equation}
If $s>t_o$, then splitting at $t_o$ yields the basic bound
\begin{equation}\label{eq:Ut-large-s}
%U_t(t_o,s)=O\big(t_o^{-\gamma}\big)
%\quad\text{for }\Gamma=0,
%\qquad
U_t(t_o,s)=O(1)\quad\text{since }\Gamma>0,
\end{equation}
uniformly over $s\in[t_o,1]$ (the case $\Gamma=\gamma$ only inserts an extra $\log(s/t_o)$, which does not affect integrability). Now, split the integral \eqref{eq:Sc-master} at $t_v=t_o$. For the forward region \(t_v<t_o\) we get, using~\eqref{eq:kappa-open} (forward), \eqref{eq:ht-bound}, and~\eqref{eq:Ut-small-s},
		\[
			\kappa_t(t_o,t_v)\,U_t(t_o,t_v)\,h_t(t_v)
			\asymp {t_v^{-\gamma} t_o^{\gamma-1}}\cdot t_o^{\gamma-1}t_v^{1-\gamma} \cdot t_v^{-\gamma}
			\asymp t_o^{-2(1-\gamma)}\,t_v^{1-3\gamma}.
		\]
		Hence, recalling \(\gamma\geq 1/2\) and \(\bar t=1/(t\log t)\),
			\begin{equation}\label{eq:Sc-pointwise-Gpos}
				S^c(t_o)\big|_{\text{forward}} \leq  C\,t_o^{-2(1-\gamma)}\int\limits_{\bar t}^{t_o} t_v^{1-3\gamma}\d t_v
                =
                \begin{cases}
				O(t_o^{-\gamma}), & \gamma<2/3,
				\\
				O\big(t_o^{-2/3}\log t\big), & \gamma=2/3,
				\\
				O\big(t_o^{-2(1-\gamma)}(t\log t)^{3\gamma-2}\big), & \gamma>2/3.
			\end{cases}
			\end{equation}
	For the reciprocal region \(t_v\geq t_o\), we similarly obtain, using~\eqref{eq:kappa-open} (reciprocal), \eqref{eq:ht-bound}, and~\eqref{eq:Ut-large-s},
		\[
			\kappa_t(t_o,t_v)\,U_t(t_o,t_v)\,h_t(t_v) \;=\; O\big(t_o^{\Gamma-\gamma}\,t_v^{-(1-\gamma+\Gamma)}\cdot t_v^{-\gamma}\big) = O\big(t_o^{\Gamma-\gamma}\,t_v^{-(1+\Gamma)}\big).
		\]
		Integrating $t_v\in[t_o,1]$ yields $O(t_o^{-\gamma})$, hence this part is never of larger order than \eqref{eq:Sc-pointwise-Gpos}.
        The desired bound of Part~(i) is then obtained by integrating \eqref{eq:Sc-pointwise-Gpos} over $t_o\in[\bar t,1]$.

\paragraph{The case $\boldsymbol{\Gamma=0}$.}
We have to argue differently in this case. Since all arcs point in both directions, the graph is effectively undirected so that the Bound~\eqref{eq:Sc_upperbound} is not precise enough. Instead, we make use of the facts that \(\gamma\geq 1/2\) and that the model is undirected so that we can write and dominate the connection probability as
\begin{equation*}
	\begin{aligned}
		\P_{\u,\v}(\u\leftrightarrow\v)
		&
			= \rho\big(\beta^{-1}(t_u\wedge t_v)^\gamma(t_u\vee t_v)^{1-\gamma}\d_t(u,v)^d\big)\leq \rho\big(\beta^{-1}t_u^\gamma t_v^{\gamma}\d_t(u,v)^d\big)=:\tilde\varphi(\u,\v),
	\end{aligned}
\end{equation*}
where we use the explicit form \(\rho(x)=1\wedge x^d\) in the following. The proof is now performed by calculating the order of closed bow-ties in the model \(\tilde\varphi\) under the birth-time truncation \(\bar t=1/(t\log t)\). Using Mecke's equation and the product structure of the birth-times in the dominating connection probability and translation-invariance, we infer
\[
	\begin{aligned}
		\int\limits_{\bar t}^1 S^c(t_o) \d t_o
		& 
			\leq \int\limits_{\bar t}^1 \d t_o \hspace{-0.3 cm}\int\limits_{(\bbT_t^d\times(\bar t,1))^3}\hspace{-0.4 cm} \d \y \d \u\d \v \ \tilde\varphi(\o,\u)\tilde\varphi(\o,\v)\tilde\varphi(\y,\u)\tilde\varphi(\y,\v)
		\\ &
			\leq C\int\limits_{\bar t}^1 \d t_o  \hspace{-0.3cm}\int\limits_{\bbT_t^d\times(\bar t,1)} \hspace{-0.3 cm} \d \y \ \Big(\hspace{-0.3cm }\int\limits_{\bbT_t^d\times(\bar t,1)} \hspace{-0.4cm} \d\u \  \rho\big(\beta^{-1}t_o^\gamma t_u^{\gamma}|u|^d\big)\rho\big(\beta^{-1}t_u^\gamma t_y^{\gamma}|y|^d\big)\Big)^2
		\\ &
			\leq C \int\limits_{\bar t}^1 \d t_o \ t_o^{-\gamma} \hspace{-0.3 cm}\int\limits_{\R^d\times(\bar t,1)} \hspace{-0.3cm}\d \y \ \Big(\int\limits_{\bar t}^1\d t_u \ t_u^{-\gamma}\big(1\wedge(t_y^\gamma t_u^\gamma |y|^d)^{-\delta}\big)\Big)^2
		\\ &
			\leq C \hspace{-0.3cm }\int\limits_{\R^d\times(\bar t,1)} \hspace{-0.3 cm} \d \y \ \Big(\int\limits_{\bar t}^1\d t_u \ t_u^{-\gamma}\big(1\wedge(t_y^\gamma t_u^\gamma |y|^d)^{-\delta}\big)\Big)^2,
	\end{aligned}
\]
where \(C>1\) does not depend on \(t\) and may have changed from line to line. For the integral under the square, we split the integration domain according to the minimum and get
\[
	\begin{aligned}
		\Big(\int\limits_{\bar t}^1 
		& 
			\d t_u \ t_u^{-\gamma}\big(1\wedge(t_y^\gamma t_u^\gamma |y|^d)^{-\delta}\big)\Big)^2
		\\ &
			\leq C\Big(\1_{\{|y|^d<t_y^{-\gamma}\}}\Theta(1) + \1_{\{t_y^{-\gamma}<t_y^{-\gamma}/\bar t\}}\Big[\int\limits_{\bar t}^{t_y^{-1}|y|^{-d/\gamma}} \hspace{-0.3 cm} \d t_u \ t_u^{-\gamma} + |y|^{-d\delta}t_y^{-\gamma\delta}\int\limits_{t_y^{-\gamma}|y|^{-d/\gamma}}^1 \hspace{-0.3 cm} \d t_u \ t_u^{-\gamma(\delta+1)}\Big]^2\Big)
		\\ &
			\leq C t_y^{2(\gamma-1)}|y|^{2d(1-1/\gamma)}.
	\end{aligned}
\]
Plugging this back into the first integral yields the upper bound
\[
	\begin{aligned}
		\int\limits_{\bar t}^1 S^c(t_o) \d t_o
		&
			\leq  C\int\limits_{\bar t}^1 \d t_y \Big[t_y^{-\gamma}+ \hspace{-0.5cm}\int\limits_{t_y^{-\gamma}\leq |y|^d\leq t_y^{-\gamma/d}/\bar t} \hspace{-0.5 cm}\d r \ \  t_y^{2(\gamma-1)/d}r^{d(3-2/\gamma)-1}\Big]
		% \\ &
			\leq \begin{cases}
				O(1), &  \gamma<\tfrac{2}{3},
				\\
				O(\log(t)), &  \gamma=\tfrac{2}{3},
				\\
				O((t\log t)^{3-2/\gamma}, & \gamma>\tfrac{2}{3},
			\end{cases}	
	\end{aligned}
\]
as claimed. The final statement follows from the simple observation that \(2\gamma-1>3-2/\gamma\) for \(\gamma<1\). This concludes the proof.
\end{proof}

\begin{lemma}\label{lem:open-bowtie-conc}
Assume $\gamma\ge \tfrac12$ and recall \(Y_t\), introduced in~\eqref{eq:defY_t} including the birth-time truncation at $1/(t\log t)$.
Then, \({Y_t}/{\E Y_t}\to1\), in probability, as \(t\to\infty\). In particular, for every fixed $\varepsilon\in(0,1)$,
\[
	\lim_{t\to\infty}\P\big(Y_t<\varepsilon\,\E Y_t\big)= 0.
\]
\end{lemma}
\begin{proof}
We work with \emph{ordered} open bow-ties and keep the birth-time truncation at $\bar t = 1/(t\log t)$. Here, ordered means that the quadruple \((\x,\u,\v,\w)\) forms an open bow-tie if \(\x\to\u\), \(\x\to\v\), and \(\w\to\v\). Recall \(S(t_o)\) introduced in~\eqref{eq:defS} and the relation
\begin{equation}\label{eq:EYt}
	\E Y_t = t\int_{\bar t}^1 S(t_o)\,\d t_o=:t\,G_t.
\end{equation}
For $\gamma\ge\frac12$, we have $G_t\to\infty$ by Lemma~\ref{lem:orderSopen}. In order to prove the lemma it suffices to prove $\Var(Y_t)=o((\E Y_t)^2)$ as \(Y_t\) is non-negative. The proof now consists of four steps.

\paragraph{Step 1: variance reduction.}
Write 
	\[
		Y_t^2=\sum_{b }\sum_{b'} \1\{b \text{ forms open bow-tie}\}\1\{b'\text{ forms open bow-tie}\},
	\] 
where $b,b'$ range over ordered quadruples $(\x,\u,\v,\w)$ and $(\x',\u',\v',\w')$ (where all vertices that belong to the same quadruple are distinct). 
Split pairs $(b,b')$ into
\begin{itemize}
\item[(D)] \emph{disjoint pairs:} the two bow-ties use $8$ distinct vertices;
\item[(O)] \emph{overlap pairs:} the two bow-ties share at least one vertex (no further restriction).
\end{itemize}
The disjoint contribution equals $(\E Y_t)^2$ by the multivariate Mecke formula, hence
\begin{equation}\label{eq:var-basic}
    \Var(Y_t) =\E[Y_t^2]-(\E Y_t)^2 \le \E Y_t + \E R_t,
\end{equation}
where $R_t$ is the total overlap-pair count.

\paragraph{Step 2: overlap pairs with the same centre ($\boldsymbol{\x=\x'}$).}
Let $R_t^{(c)}$ be the overlap contribution coming from pairs with $\x=\x'$.
Applying Mecke with one distinguished centre vertex gives
\[
\E R_t^{(c)}
\le C \,t\int\limits_{\bar t}^1 \E_{\0}\Big[ \Big(\sum_{\substack{\u,\v,\w}}
\1\{\0\to\u\}\1\{\0\to\v\}\1\{\w\to\v\}\Big)^2\Big]\, \d t_o,
\]
for a constant $C$ accounting for the finite number of ordered ways in which two bow-ties can overlap while
sharing the same centre\footnote{Here we use that we do \emph{not} exclude any overlaps across the two bow-ties;
the factor $C$ depends only on the (finite) set of role-identification patterns.}.
Expanding the square and using $\1\{\cdot\}^2=\1\{\cdot\}$ and $\1\{\cdot\}\le1$ yields
\[
\E_{\0}\Big[\Big(\sum_{\substack{\u,\v,\w}}
\1\{\0\to\u\}\1\{\0\to\v\}\1\{\w\to\v\}\Big)^2\Big]
\le C\,\big(S(t_o)^2+S(t_o)\big),
\]
and hence
\begin{equation}\label{eq:center-overlap}
\E R_t^{(c)}\le C\,t\int\limits_{\bar t}^1 S(t_o)^2+S(t_o)\, \d t_o
\end{equation}
where \(C\) may have changed.

\paragraph{Step 3: overlap pairs with different centres ($\boldsymbol{\x\neq\x'}$).}
Let $R_t^{(d)}$ be the overlap contribution from pairs with $\x\neq\x'$.
In any such overlap pair there exists at least one shared vertex among
$\{\u,\v,\w\}\cap\{\u',\v',\w'\}$ or $\{\u,\v,\w\}\cap\{\x'\}$ or $\{\u',\v',\w'\}\cap\{\x\}$.
Fix one such shared role pattern. Write down its Mecke integral over the \emph{distinct} space--time points
appearing in the union. In the integrand we have a product of (at most) six arc indicators.
Whenever two indicators involve arcs \emph{into the same shared vertex} (e.g.\ $\x\to \z$ and $\x'\to \z$ for the shared
$\z$), we use the pointwise bound $\P_{\x,\z}(\x\to\z) \P_{\x',\z}(\x'\to\z)\le \P_{\x,\z}(\x\to\z)$ to drop one factor.
Doing this once per pattern ensures that, after integrating spatial variables, the overlap integral is bounded by
a constant times a product of two \emph{open-bow-tie first-moment integrals} (one rooted at $\x$, one rooted at $\x'$),
with no squared spatially-integrated kernels appearing. Concretely, we obtain for each fixed overlap pattern 
\[
\text{(pattern contribution)} \;\le\; C\, t\int\limits_{\bar t}^1 \int\limits_{\bar t}^1 \!\! S(t_x)\,S(t_{x'})\,\d t_x\,\d t_{x'}
\;=\; C\,t\,G_t^2,
\]
where the prefactor $t$ is the single free spatial translation (the union is connected through the shared vertex)
and the remaining spatial integrals are bounded using the same estimate as in \eqref{eq:spaceint-open}.
Summing over the finitely many ordered overlap patterns yields
\begin{equation}\label{eq:diff-center-overlap}
\E R_t^{(d)} \le C\,t\,G_t^2.
\end{equation}

We infer combining \eqref{eq:center-overlap} and \eqref{eq:diff-center-overlap}
\begin{equation}\label{eq:ERt-bound}
\E R_t \le C\,t\,G_t^2 + C\,t\int\limits_{\bar t}^1 \big(S(t_o)^2+S(t_o)\big)\,\d t_o.
\end{equation}

\paragraph{Step 4: $\boldsymbol{\Var(Y_t)=o((\E Y_t)^2)}$.}
By \eqref{eq:var-basic}, \eqref{eq:EYt}, and \eqref{eq:ERt-bound},
\[
\frac{\Var(Y_t)}{(\E Y_t)^2}
\le
\frac{\E Y_t}{(\E Y_t)^2}
+
\frac{C\,t\,G_t^2}{t^2G_t^2}
+
\frac{C\,t\int S(t_o)^2\d t_o}{t^2G_t^2}
+
\frac{C\,t\int S(t_o)\d t_o}{t^2G_t^2}.
\]
The first and last terms are $O(1/\E Y_t)$ and $O(1/(tG_t))$, hence $o(1)$ since $G_t\to\infty$.
The second term equals $C/t=o(1)$.
It therefore remains to bound $\int S^2$.
Let $S_{\max}:=\sup_{t_o\in[\bar t,1]} S(t_o)$. Then
\[
	\int\limits_{\bar t}^1 S(t_o)^2\,\d t_o \le S_{\max}\int\limits_{\bar t}^1 S(t_o)\,\d t_0 = S_{\max}\,G_t,
\]
implying
\[
	\frac{t\int S(t_o)^2\d t_o}{t^2 G_t^2}\le \frac{S_{\max}}{tG_t}=\frac{S_{\max}}{\E Y_t}.
\]
Finally, for $\gamma\ge\frac12$, one has $S_{\max}=o(\E Y_t)$ since~\eqref{eq:orderB(to)} in the proof of Lemma~\ref{lem:orderSopen} yields $S(t_o)=O(t^{2\gamma-1}t_o^{\gamma-1}(\log t)^2)$, hence $S_{\max}=O(t^{\gamma}(\log t)^{1+\gamma})$, whereas $\E Y_t=tG_t\asymp t(\log t)^{1+\1\{\Gamma=0\}}$ if $\gamma=1/2$, and or is of order
$\Theta(t^{2\gamma}(\log t)^{2\gamma+\1\{\Gamma=0\}}$ if $\gamma>1/2$.
Thus $S_{\max}/\E Y_t\to0$ and particularly $\Var(Y_t)/(\E Y_t)^2\to0$. Hence, by Chebyshev, 
$Y_t/(\E Y_t)\to 1$ in probability, %and in particular $\P(Y_t<\varepsilon \E Y_t)\to0$ for every $\varepsilon\in(0,1)$. 
concluding the proof.
\end{proof}
\begin{lemma}\label{lem:orderMN}
	Let \(\gamma\geq 1/2\) and consider the quantities \(\bfM_t\) and \(\bfN_t\), defined in~\eqref{eq:defMN}. We have
	\begin{equation*}
	\E[\bfM_t]\asymp 
	\begin{cases}
		t \log(t), & \text{ if }\gamma = 1/2,
		\\
		t^{2\gamma}(\log t)^{2\gamma-1}, & \text{ if } \gamma>1/2.
	\end{cases}
\end{equation*}
Moreover, \(\E\bfN_t =o(\E\bfM_t)\). 	
\end{lemma}
\begin{proof}
	Recall that all sums over pairs always relate to \emph{ordered} pairs and further recall the birth-time truncation at \(\bar t = 1/(t\log t)\). For $(\x,\y)\in \D^t\times \D^t$, we define
\[
	\C_t(\x,\y):=\sharp\big(\scrN^\text{out}_t(\x)\cap \scrN^\text{out}_t(\y)\big),\quad \text{ and }\quad  \mathbf{n}_t(\x):=\sharp\scrN^\text{out}_t(\x).
\]
A direct count of ordered pairs shows that, for a given vertex \(\x\in\D^t\) with $\{\mathbf{n}_t(\x)\geq 2\}$ and \(\y\in\scrN_t^\text{out}(\x)\),
\[
	\sum_{\substack{\u,\v\in \D_t}} \mathbbm{1}\{\y\to\u,\y\to\v\}\, \mathbbm{1}\{\x\to\u,\x\to\v\} = \C_t(\x,\y)\big(\C_t(\x,\y)-1\big),
\]
and
\[
	\sum_{\substack{\u,\v\in \D_t}}\mathbbm{1}\{\y\to\u\ \text{or}\ \y\to\v\}\,\mathbbm{1}\{\x\to\u,\x\to\v\} = \big(2\mathbf{n}_t(\x)-\C_t(\x,\y)-1\big)\,\C_t(\x,\y).
\]
Consequently, whenever $(\x,\y)\in\scrI_t$
\begin{equation}\label{eq:locIC-algebra}
	\mathrm{c}^{\mathrm{ic}}((\x,\y),\D^t)
	%=
	%	\frac{2\C_t(\x,\y)\big(\C_t(\x,\y)-1\big)}{\big(2\mathbf{n}_t(\x)-\C_t(\x,\y)-1\big)\,\C_t(\x,\y)}
	=
		\frac{2\big(\C_t(\x,\y)-1\big)}{2\mathbf{n}_t(\x)-\C_t(\x,\y)-1}
	\le 
		2\,\mathbbm{1}\{\C_t(\x,\y)\ge2\}.
\end{equation}
Define the (ordered) closed bow--tie count
\[
	\cZ_t:=\sum_{\x,\y\in \D^t} \C_t(\x,\y)\big(\C_t(\x,\y)-1\big),
\]
so that $\bfN_t\le 2 \cZ_t$ by \eqref{eq:locIC-algebra}. Now, recall the birth-time truncation and obtain by Mecke's formula,
\begin{equation*}%\label{eq:Mt-Mecke}
	\E[\bfM_t]=t\int_{\bar t}^1 J(t_o)\,\d t_o,
	\quad \text{ where }
	J(t_o):=\E_{(o,t_o)} \sum_{\y\in \D^t} \mathbbm{1}\{\y\in \scrI_t\}.
\end{equation*}
Given $\y$, the common out-neighbours of $(o,t_o)$ and $\y$
form a Poisson process with mean
\[
	\eta_t\big((o,t_o),\y\big)	= \int \P_{\y,\v}(\y\to\v)\,\lambda_{(o,t_o)}^{\mathrm{out}}(d\v),
\]
and the remaining out--neighbours of $(o,t_o)$ are independent by the Poisson thinning property. Here, we recall the definition of the intensity measure \(\lambda_{(o,t_o)}^\text{out}\) from Section~\ref{sec:degrees}. In particular, for $\mu(t_o):=\int\lambda_{(o,t_o)}^{\mathrm{out}}(d\v)=\E_{(o,t_o)}[\mathbf n_t((o,t_o))]$ and $\eta_t:=\eta_t((o,t_o),\y)\le \mu(t_o)$, we have
\[
	\P_{(o,t_o),\y}(\y\in\scrI_t) = \P_{\o,\y}\big(\mathbf n_t((o,t_o))\geq 2,\ \C_t((o,t_o),\y)\geq 1) = 1-e^{-\eta_t}-\eta_t e^{-\mu(t_o)}.
\]
Using that $\mu(t_o)$ is bounded away from $0$ on $t_o\in[\bar t,1]$  and that $0\le \eta_t\le\mu(t_o)$, one obtains %constants $0<c_1\le c_2<\infty$ such that
\[
	1-e^{-\eta_t}-\eta_t e^{-\mu(t_o)} \asymp \eta_t
\qquad\text{for all }t_o\in[\bar t,1]\text{ and all }\y\in \D^t.
\]
Recalling the quantity $B(t_o)$ from~\eqref{eq:defB(to)} and its asymptotic behaviour~\eqref{eq:orderB(to)} and integrating with respect to $\y$ therefore yields $J(t_o)\asymp B(t_o)$ and hence
\begin{equation*}%\label{eq:EMt-scale}
	\E[\bfM_t]\asymp t\int\limits_{\bar t}^1 B(t_o)\,\d t_o \asymp 
	\begin{cases}
		t \log t, & \text{ if }\gamma = 1/2,
		\\
		t^{2\gamma}(\log t)^{2\gamma-1}, & \text{ if } \gamma>1/2,
	\end{cases}
\end{equation*}
as claimed. For the second claim, we use that, given $(\x,\y)$, the variable $\C_t(\x,\y)$ is Poisson with mean $\eta_t(\x,\y)$ as above, so $\E_{\x,\y}[\C_t(\x,\y)(\C_t(\x,\y)-1)]=\eta_t(\x,\y)^2$.
Therefore, 
\begin{equation*}%\label{eq:EZt-Mecke}
	\E[\cZ_t] = t\int\limits_{\bar t}^1 S^c(t_o)\,\d t_o,
\end{equation*}
where $S^c(t_o)$ is given in~\eqref{eq:defSclosed}. The estimates on $S^c$ of Lemma~\ref{lem:orderSclosed} imply that \(\E[\cZ_t]=o\big(\E[M_t]\big)\).
As \(\bfN_t\leq 2\cZ_t\), this concludes the proof.
\end{proof}
\end{document}